\documentclass{article}

\usepackage{latexsym}
\usepackage{amsmath,amssymb,amstext,amsthm}
\usepackage{graphicx}
\usepackage{stmaryrd}
\usepackage{marginnote}
\usepackage{comment}
\usepackage{tikz}
\colorlet{dblue}{blue!40!black}
\usepackage[colorlinks,linkcolor=dblue,citecolor=dblue,urlcolor=dblue]{hyperref}

\newcommand{\texcomp}[3]{#1{#2{#3}}}
\newcommand{\prim}[1]{#1'}

\theoremstyle{plain}
\newtheorem{theorem}{Theorem}[section]
\newtheorem{corollary}[theorem]{Corollary}
\newtheorem{lemma}[theorem]{Lemma}
\newtheorem{proposition}[theorem]{Proposition}

\theoremstyle{definition}
\newtheorem{definition}[theorem]{Definition}
\newtheorem{remark}[theorem]{Remark}
\newtheorem{example}[theorem]{Example}

\newtheorem{claim}{Claim}

\newtheorem{case}{Case}

\newtheorem{aplemma}{Lemma}[section]
\newcommand{\apsection}[1]{
  \let\oldthesection\thesection
  \renewcommand{\thesection}{Appendix \oldthesection}
  \section{#1}
  \let\thesection\oldthesection
}

\newcommand{\mrm}{\mathrm}

\newcommand{\mcl}{\mathcal}

\newcommand{\mbf}{\mathbf}
\newcommand{\mbs}{\boldsymbol}
\newcommand{\mbb}{\mathbb}
\newcommand{\mfrk}{\mathfrak}

\newcommand{\nb}{\nobreakdash}

\newcommand{\punc}[1]{\:\text{#1}}

\newcommand{\dquad}{\quad\quad}

\newcommand{\id}[1]{#1}
\newcommand{\del}[1]{}

\newcommand{\sub}[2]{#1_{#2}}
\newcommand{\super}[2]{#1^{#2}}
\newcommand{\binap}[3]{#2\mathbin{#1}#3}
\newcommand{\funap}[2]{#1(#2)}

\newcommand{\funnap}[3]{\funap{\super{#1}{#2}}{#3}}

\newcommand{\bfunap}[3]{\funap{#1}{#2,#3}}

\newcommand{\swhere}{{|}}
\newcommand{\where}{\mathrel{\swhere}}
\newcommand{\juxt}[2]{#1#2}
\newcommand{\sfunin}{{:}}
\newcommand{\funin}{\mathrel{\sfunin}}

\newcommand{\displset}[2]{\left\{\,#1 \,\, \middle\vert \,\, #2\,\right\}}

\newcommand{\smaxx}{\mrm{max}}
\newcommand{\maxx}[1]{\juxt{\smaxx}{\,#1}}
\newcommand{\eq}{=}
\newcommand{\nequiv}{\not\equiv}

\newcommand{\myall}[2]{(\forall#1)(#2)}
\newcommand{\myex}[2]{(\exists#1)(#2)}

\newcommand{\summ}[3]{\sum_{#1}^{#2} #3}

\newcommand{\lt}{<}
\newcommand{\gt}{>}

\renewcommand{\emptyset}{{\varnothing}}
\newcommand{\card}{\length}

\newcommand{\infinity}{\infty}

\newcommand{\ssetfun}{{\to}}
\newcommand{\setfun}{\binap{\ssetfun}}
\newcommand{\setop}[1]{\setfun{#1}{#1}}

\newcommand{\domrestr}[2]{#1{|}_#2}

\newcommand{\congrclasses}[1]{\Z/#1\Z}
\newcommand{\mgroup}{\sub{\Z}}
\newcommand{\dlog}{\sub{\lambda}}

\newcommand{\residues}[2]{[#1]_{#2}}

\newcommand{\ssymgroup}{\mfrk{S}}
\newcommand{\symgroup}{\sub{\ssymgroup}}

\newcommand{\length}[1]{|#1|}

\newcommand{\pairlft}{{\langle}}
\newcommand{\pairrgt}{{\rangle}}
\newcommand{\pairsep}{{,\,}}
\newcommand{\pairstr}[1]{\pairlft#1\pairrgt}
\newcommand{\pair}[2]{\pairstr{#1\pairsep#2}}
\newcommand{\triple}[2]{\pair{#1\pairsep#2}}

\newcommand{\nat}{\mbb{N}}
\newcommand{\natpos}{\sub{\nat}{{\gt0}}}
\newcommand{\posnat}{\natpos}

\newcommand{\Z}{\mathbb{Z}}

\newcommand{\nth}{\funap}

\newcommand{\bit}{\nalph{2}}

\newcommand{\inv}{\overline}

\newcommand{\str}[1]{\super{#1}{\nat}}
\newcommand{\strp}[1]{\super{#1}{\posnat}}

\newcommand{\astr}{\sigma}
\newcommand{\iastr}{\sub{\astr}}
\newcommand{\bstr}{\tau}

\newcommand{\llist}[1]{\super{#1}{\infinity}}

\newcommand{\aalph}{\Sigma}

\newcommand{\wrd}[1]{\super{#1}{\ast}}
\newcommand{\wrdemp}{\varepsilon}
\newcommand{\newrd}[1]{\super{#1}{{+}}}
\newcommand{\awrd}{w}

\newcommand{\cat}{\juxt}

\newcommand{\sreverse}{R}
\newcommand{\reverse}[1]{\super{#1}{\sreverse}}

\newcommand{\thuemorse}{\mbox{Thue\hspace{.08em}--Morse}}

\newcommand{\pd}{period doubling}

\newcommand{\Sierpinski}{Sierpi\'{n}ski}

\newcommand{\sdiff}{\Delta}
\newcommand{\diff}{\funap{\sdiff}}
\newcommand{\sTdiff}{\Delta_{\stoeplitz}}
\newcommand{\Tdiff}{\funap{\sTdiff}}
\newcommand{\sTsum}{\int_{\stoeplitz}}
\newcommand{\Tsum}{\funap{\sTsum}}

\newcommand{\tsp}{\hspace{.08em}}
\newcommand{\ntsp}{\hspace{-.08em}}

\newcommand{\asub}[3]{#1_{#2,#3}}

\newcommand{\sdiv}{{/}}
\renewcommand{\div}[2]{#1\sdiv#2}

\newcommand{\ssimilar}{{\thicksim}}

\newcommand{\similarx}{\mathrel{\ssimilar}}

\newcommand{\dragon}{\mbs{f}}
\newcommand{\PF}{\dragon}

\newcommand{\PD}{\mbs{p}}

\newcommand{\morse}{\mbs{m}}
\newcommand{\terdragon}{\mbs{t}}
\newcommand{\hanoi}{\mbs{h}}
\newcommand{\mephisto}{\mbs{w}}

\newcommand{\sierpinski}{\mbs{s}}

\newcommand{\primes}{\mbf{P}}

\newcommand{\sspadval}{v}
\newcommand{\spadval}{\sub{\sspadval}}
\newcommand{\padval}[1]{\funap{\spadval{#1}}}

\newcommand{\Tsubst}[2]{#1[#2]}

\newcommand{\sTcmp}{{\circ}}
\newcommand{\Tcomp}{\binap{\sTcmp}}
\newcommand{\Tcmpx}{\mathbin{\sTcmp}}

\newcommand{\Tpow}[2]{\super{#1}{(#2)}}

\newcommand{\stoeplitz}{T}
\newcommand{\toeplitz}{\funap{\stoeplitz}}
\newcommand{\toeplitzi}[1]{\funap{\sub{\stoeplitz}{#1}}}

\newcommand{\stoeplitzeq}{\sub{{=}}{\stoeplitz}}
\newcommand{\toeplitzeq}{\mathrel{\stoeplitzeq}}

\newcommand{\gap}{{?}}

\newcommand{\fgap}{\id}
\newcommand{\agap}{\fgap{f}}
\newcommand{\bgap}{\fgap{g}}
\newcommand{\cgap}{\fgap{h}}
\newcommand{\apat}{P}
\newcommand{\bpat}{Q}

\newcommand{\rotgap}[1]{{\super{\gap}{{+}#1}}}
\newcommand{\galph}[1]{#1\cup\symgroup{#1}}
\newcommand{\gaalph}{\galph{\aalph}}

\newcommand{\nrofin}[2]{\sub{\length{#2}}{#1}}
\newcommand{\gapnr}{\nrofin{\gap}}

\newcommand{\oeis}[1]{\cite[\href{http://oeis.org/classic/#1}{\small{\texttt{#1}}}]{sloa:2010}}

\newcommand{\nalph}{\sub{\aalph}}

\newcommand{\sAS}{\mcl{AS}}
\newcommand{\siAS}{\super{\sAS}}
\newcommand{\AS}{\funap{\sAS}}
\newcommand{\iAS}{\texcomp{\funap}{\siAS}}

\newcommand{\ampy}{\mu}
\newcommand{\iampy}{\sub{\ampy}}

\newcommand{\spgs}{\mbs{\sspadval}}
\newcommand{\pgs}[1]{\sub{\spgs}{\ntsp#1}}

\newcommand{\sKmult}{{\times}}
\newcommand{\Kmult}{\binap{\sKmult}}
\newcommand{\Kpow}[2]{\super{#1}{(#2)}}

\newcommand{\skeane}{K}
\newcommand{\keane}{\funap{\skeane}}
 
\newcommand{\ablk}{u}

\newcommand{\bblk}{v}

\newcommand{\slcm}{\mrm{lcm}}
\newcommand{\lcm}{\bfunap{\slcm}}
\newcommand{\sgcd}{\mrm{gcd}}
\renewcommand{\gcd}{\bfunap{\sgcd}}

\newcommand{\sfuncomp}{{\circ}}
\newcommand{\funcomp}{\binap{\sfuncomp}}

\newcommand{\nbd}{\nobreakdash}
\newcommand{\hyph}[2]{#1\nbd-#2}

\newcommand{\ssetxor}{\Delta}

\newcommand{\setxorx}{\mathbin{\ssetxor}}

\newcommand{\abs}[1]{|#1|}

\newcommand{\gtimes}[2]{#1 \odot #2}
\newcommand{\ghom}[1]{\odot #1}
\newcommand{\gscale}{\bfunap{\xi}}

\newcommand{\ssign}{\mrm{sign}}
\newcommand{\sign}{\funap{\ssign}}

\title{Arithmetic Self-Similarity of Infinite Sequences~%
  \footnote{This research has been partially funded by 
    the Netherlands Organisation for Scientific Research (NWO)
    under grant numbers 612.000.934 and 639.021.020.%
  }
}

\author{Dimitri Hendriks 
  \and Frits G.W. Dannenberg
  \and J\"{o}rg Endrullis \\[.5ex]
  \and Mark~Dow
  \and Jan Willem Klop
}
\begin{document}

\maketitle

\begin{abstract}
  \noindent
  We define the arithmetic self-similarity (AS) of 
  a one-sided infinite sequence~$\astr$ 
  to be the set of arithmetic subsequences of~$\astr$ 
  which are a vertical shift of $\astr$.
  We study the AS 
  of several families of sequences,
  viz.\ completely additive sequences, Toeplitz words 
  and Keane's generalized Morse sequences.
  We give a complete characterization of the AS of completely additive sequences, 
  and classify the set of single-gap Toeplitz patterns 
  that yield completely additive Toeplitz words.
  We show that every arithmetic subsequence of a Toeplitz word 
  generated by a one-gap pattern is again a Toeplitz word.
  Finally, we establish that generalized Morse sequences 
  are specific sum-of-digits sequences,
  and show that their first difference is a Toeplitz word.
  
\end{abstract}

\section{Introduction}
Some infinite sequences are similar to a part of themselves. 
Zooming in on a part of the structure reveals the whole structure again. 
Of special interest are what we may call \emph{scale-invariant} sequences.
A sequence $\mbs{w} = (w(n))_{n \ge 1}$ 
(over some additive group~$\aalph$)
is scale-invariant if for all `dilations' $k\geq1$ 
the compressed sequence 
$\div{\mbs{w}}{k} = (w(kn))_{n \ge 1}$ 
is a `vertical shift' of $\astr$, that is,
$w(kn) - w(n)$ is constant.
An example of a scale-invariant sequence is the \pd{} sequence~$\PD$ 
which can be defined as $\PD = (\padval{2}{n} \bmod{2})_{n\ge1}$
with $\padval{2}{n}$ the $2$-adic valuation of $n$.
We have that $\div{\PD}{k} = \PD + \PD(k) \bmod 2$ for every $k\geq1$:
\newcommand{\z}{\phantom}
  \begin{align*}
  \PD 
  & \,\,\,=\,\,\,
  0\,1\,0\,0\,0\,1\,0\,1\,0\,1\,0\,0\,0\,1\,0\,0\,
  0\,1\,0\,0\,0\,1\,0\,1\,0\,1\,0\,0\,0\,1\,0\,1\,
  0\,1\,0\,0\,0\,1\,0\,1\,0\,1\,0\,
  \ldots
  \\
  \div{\PD}{2}
  & \,\,\,=\,\,\,
  \z{0}\,1\,\z{0}\,0\,\z{0}\,1\,\z{0}\,1\,\z{0}\,1\,\z{0}\,0\,\z{0}\,1\,\z{0}\,0\,
  \z{0}\,1\,\z{0}\,0\,\z{0}\,1\,\z{0}\,1\,\z{0}\,1\,\z{0}\,0\,\z{0}\,1\,\z{0}\,1\,
  \z{0}\,1\,\z{0}\,0\,\z{0}\,1\,\z{0}\,1\,\z{0}\,1\,\z{0}\,
  \ldots
  \\
  \div{\PD}{3}
  & \,\,\,=\,\,\,
  \z{0\,1}\,0\,\z{0\,0}\,1\,\z{0\,1}\,0\,\z{1\,0}\,0\,\z{0\,1}\,0\,\z{0\,
  0}\,1\,\z{0\,0}\,0\,\z{1\,0}\,1\,\z{0\,1}\,0\,\z{0\,0}\,1\,\z{0\,1}\,
  0\,\z{1\,0}\,0\,\z{0\,1}\,0\,\z{1\,0}\,1\,\z{0}\,
  \ldots
  \\
  \div{\PD}{5}
  & \,\,\,=\,\,\,
  \z{0\,1\,0\,0}\,0\,\z{1\,0\,1\,0}\,1\,\z{0\,0\,0\,1}\,0\,\z{0\,
  0\,1\,0}\,0\,\z{0\,1\,0\,1}\,0\,\z{1\,0\,0\,0}\,1\,\z{0\,1\,}
  \z{0\,1}\,0\,\z{0\,0\,1\,0}\,1\,\z{0\,1\,0}\,
  \ldots
\end{align*}
%
Scale-invariant sequences $\astr$ with first term $\astr_1 = 0$ 
are known as \emph{completely additive} sequences, 
that is, sequences~$\astr$ such that 
$\nth{\astr}{nm} = \nth{\astr}{n} + \nth{\astr}{m}$ for all positive integers $n,m$.


In general, given an infinite sequence~$\astr$ one may wonder 
what are the arithmetic subsequences of $\astr$ similar to $\astr$ itself.
This leads to the notion of what we call `arithmetic self-similarity':
For an equivalence relation~$\ssimilar$ on $\aalph^\omega$,
the \emph{arithmetic self-similarity} 
of a sequence $\astr\in\aalph^\omega$,
which we denote by~$\iAS{\ssimilar}{\astr}$, 
is the set of pairs $\pair{a}{b}$ 
such that the subsequence of $\astr$ 
indexed by the arithmetic progression $a + bn$
is equivalent to $\astr$\,: 
\[
  \iAS{\ssimilar}{\astr} = \displset{\pair{a}{b}}{\asub{\astr}{a}{b} \similarx \astr}
  \dquad 
  \text{where $\asub{\astr}{a}{b} = \astr(a) \, \astr(a+b) \, \astr(a+2b) \, \ldots$} 
\]
For instance, in \cite{endr:hend:klop:2011} it is shown that 
the arithmetic self-similarity of the Thue--Morse sequence 
$\morse = 0110 1001 \cdots$~\cite{allo:shal:1999}
with respect to `transducer-equivalence' $\diamond$ 
is the full space $\iAS{\diamond}{\morse} = \nat\times\posnat$,
i.e., for every arithmetic subsequence $\asub{\morse}{a}{b}$ of $\morse$ 
there is a finite state transducer 
which reconstructs $\morse$ from $\asub{\morse}{a}{b}$.

In this paper we focus on cyclic groups $\triple{\aalph}{{+}}{0}$ 
and on a particular equivalence~$\ssimilar$ on sequences over $\aalph$,
namely $\astr \similarx \bstr$ if and only if $\astr = \bstr +  c$ for some $c \in \aalph$.
%
%
%
Our main results are as follows:
\begin{itemize}
  
  \item
    We give a complete characterization of the arithmetic self-similarity 
    of completely additive sequences~$\astr\in\strp{\aalph}$;
    we show that 
    $\pair{a}{b} \in \AS{\astr}$ if and only if $a = b$.

  \item
    We define Toeplitz patterns to be finite words over $\gaalph$,
    where $\symgroup{\aalph}$ is the set of permutations over $\aalph$ 
    which play the role of `gaps'. 
    The composition operation defined on Toeplitz patterns 
    (Definition~\ref{def:toeplitz}) forms a monoid (Proposition~\ref{prop:Tpat:monoid}).

  \item 
    We give a complete characterization of one-gap Toeplitz patterns 
    that yield completely additive sequences (Theorem~\ref{thm:superthm}).
    These patterns are constructed using discrete logarithms.
    The completely additive sequences they generate are determined by infinite sets of primes.
    
  \item
    We show that every arithmetic subsequence of 
    a Toeplitz word generated by a one-gap pattern
    is again a Toeplitz word 
    (Theorem~\ref{thm:arith:toeplitz}).
    
%
        
  \item 
    We prove that the first difference of a generalized Morse sequence~\cite{kean:1968} 
    is a Toeplitz word (Theorem~\ref{thm:diff:keane}).
    This gives rise to an embedding from Keane's monoid of block products
    (extended to additive groups) to that of Toeplitz pattern composition 
    (Theorem~\ref{thm:Tdiff:embedding}).

  \item 
    We show how generalized Morse sequences~\cite{kean:1968} 
    are specific sum-of-digits sequences (Theorem~\ref{thm:keane:digits}).
      
  \item 
    For the \thuemorse{} sequence~$\morse = 0110 1001 \ldots$ we show that 
    $\morse(a + bn) = \morse(n) + \morse(a) \bmod{2}$ ($n\in\nat$) 
    if and only if $0 \le a \lt b = 2^m$ for some $m\in\nat$
    (Theorem~\ref{thm:AS:morse}). 

\end{itemize}

\section{Basic Definitions}\label{sec:prelims} 
We write $\nat = \{0,1,2,\ldots\}$ for the set of natural numbers,
and $\posnat$ for the set of positive integers,
$\posnat = \nat \setminus \{0\}$.
For $k \geq 0$ we define $\nalph{k} = \{0,\ldots,k-1\}$.
The sets of finite words and finite non-empty words over an alphabet~$\aalph$
are denoted by $\wrd{\aalph}$ and $\newrd{\aalph}$, respectively. 
We identify a (one-sided) infinite word 
with a map from $\nat$ to $\aalph$, 
or, when this is more convenient, from $\posnat$ to $\aalph$.
We write $\str{\aalph}$, or $\strp{\aalph}$, 
for the set of infinite words over $\aalph$:
\begin{align*}
  \str{\aalph} = \displset{\astr}{\astr\funin\nat\to\aalph}
  &&
  \strp{\aalph} = \displset{\astr}{\astr\funin\posnat\to\aalph}
\end{align*}
We let $\llist{\aalph}$ denote the set 
of all finite and infinite words over $\aalph$.
We write $\wrdemp$ for the empty word, and $\cat{x}{y}$
for concatenation of words $x\in\wrd{\aalph}$, $y\in\llist{\aalph}$.
For $u\in\wrd{\aalph}$ we let $u^0 = \wrdemp$ and $u^{k+1} = u u^k$,
and, for $u\in\newrd{\aalph}$ we write $u^\omega$ for the infinite sequence $uuu\ldots$.
For a word~$w\in\llist{\aalph}$ we use $\nth{w}{n}$ 
to denote the value at index $n$ (if defined).
The \emph{reversal} $w^R$ of $w \in \wrd{\aalph}$ is defined by 
$\wrdemp^R = \wrdemp$ and $(au)^R = u^R a$,
for all $a\in\aalph$ and $u\in\wrd{\aalph}$.

We shall primarily deal with infinite words over finite cyclic groups.
A cyclic group is a group generated by a single element (or its inverse).
Every finite cyclic group of order~$k$ is isomorphic to the additive group 
$\triple{\nalph{k}}{{+}}{0}$ with $+$ denoting addition modulo~$k$.

Let $\aalph = \triple{\aalph}{{+}_{\aalph}}{0_{\aalph}}$ 
be a finite cyclic group.
Let $\awrd\in\llist{\aalph}$, and $a\in\aalph$.
Then $\awrd +_{\aalph} a$ is defined by 
$\nth{(\awrd +_{\aalph} a)}{n} = \nth{\awrd}{n} +_{\aalph} a$, 
whenever $\nth{\awrd}{n}$ is defined.
Furthermore, we write $\astr +_{\aalph} \bstr$ 
for the sequence obtained from $\astr,\bstr\in\str{\aalph}$ 
by pointwise addition: 
$\nth{(\astr +_{\aalph} \bstr)}{n} = \nth{\astr}{n} +_{\aalph} \nth{\bstr}{n}$ 
for all $n \in \nat$;
so $\astr +_{\aalph} a = \astr +_{\aalph} a^\omega$.
When the group~$\aalph$ is clear from the context,
we will write just $+$ and $0$ for $+_{\aalph}$ and $0_{\aalph}$.

For $k \in \nalph{n}$ (or $k \in \nat$) and $c \in \nalph{m}$ (with $n, m \in \posnat$),
we write $\gtimes{k}{c}$ for the letter $c + c + \cdots + c$ ($k$ times) in the alphabet $\nalph{m}$.
For words $w = a_0 \ldots a_r \in \wrd{\nalph{n}}$ we use
$\gtimes{w}{c}$ to denote the word $(\gtimes{a_0}{c})\ldots(\gtimes{a_r}{c}) \in \nalph{m}^*$.
Likewise for infinite sequences $\astr = a_0 a_1 \ldots \in \str{\nalph{n}}$ we write
$\gtimes{\astr}{c}$ for 
$(\gtimes{a_0}{c})(\gtimes{a_1}{c})\ldots \in \str{\nalph{m}}$.
Finally, by the partial application
$\ghom{c}$ we denote the map from $\nalph{n}$ to $\nalph{m}$
defined by $k \mapsto \gtimes{k}{c}$ for all $k \in \nalph{n}$.
For $n, m \in \posnat$ we define 
\[
  \gscale{n}{m} = \frac{m}{\gcd{n}{m}} = \frac{\lcm{n}{m}}{n}
\]

\begin{lemma}\label{lem:ghom}
  Let $c \in \nalph{m}$
  s.t.\ $c$ is divisible by $\gscale{n}{m}$,
  i.e.,
  $c = \gtimes{\gscale{n}{m}}{c'}$
  for some $c'\in\nalph{m}$.
  Then the map $\ghom{c}$ is a group homomorphism 
  from $\triple{\nalph{n}}{{+_{\nalph{n}}}}{0_{\nalph{n}}}$ to $\triple{\nalph{m}}{{+_{\nalph{m}}}}{0_{\nalph{m}}}$.
\end{lemma}
\begin{proof}
  Obviously $\gtimes{0_{\nalph{n}}}{c} = 0_{\nalph{m}}$.
  Note that 
  $\gtimes{n}{c} 
   = \gtimes{n}{(\gtimes{\gscale{n}{m}}{c'})} 
   = \gtimes{(n \cdot \gscale{n}{m})}{c'}
   = 0_{\nalph{m}}$
  since $n \cdot \gscale{n}{m}$ is a multiple of $m$.
  Let $k, \ell \in \nalph{n}$ arbitrary.
  If $k +_\nat \ell < n$ 
  (here $+_\nat$ indicates that we use addition on the natural numbers),
  then trivially $\gtimes{(k +_{\nalph{n}} \ell)}{c} = \gtimes{k}{c} +_{\nalph{m}} \gtimes{\ell}{c}$.
  Otherwise $k +_\nat \ell = n +_\nat u$ for some $u \in \nalph{n}$
  and 
  $\gtimes{k}{c} +_{\nalph{m}} \gtimes{\ell}{c} 
   = \gtimes{n}{c} +_{\nalph{m}} \gtimes{u}{c}
   = \gtimes{u}{c}
   = \gtimes{(k +_{\nalph{n}} \ell)}{c}$.
\end{proof}

Let $\astr\in\str{\aalph}$ be an infinite sequence, 
and let $a\ge 0$ and $b \ge 1$ be integers. 
The \emph{arithmetic subsequence} $\asub{\astr}{a}{b}$ of $\astr$ is defined by%
  \footnote{%
    For infinite sequences $\astr\in\strp{\aalph}$ (indexed over the positive integers)
    we require both $a$ and $b$ to be positive integers, and then
    $\nth{\asub{\astr}{a}{b}}{n} = \nth{\astr}{a + b(n-1)}$ for all $n\in\posnat$.%
  }%
\begin{align*}
  \nth{\asub{\astr}{a}{b}}{n} = \nth{\astr}{a + bn} && (n\in\nat)
\end{align*}


By composition of arithmetic progressions we have 
$\asub{(\asub{\astr}{a}{b})}{c}{d} = \asub{\astr}{a+bc}{bd}$ for all $a,c\ge 0$ and $b,d\ge 1$.

A sequence $\astr\in\str{\aalph}$ is \emph{ultimately periodic}
if there are $n_0\in\nat$ and $t\in\posnat$ such that 
$\astr(n+t) = \astr(n)$ for all $n \ge n_0$.
Here $t$ is called the \emph{period} 
and $n_0$ the \emph{preperiod}.
If $n_0 = 0$, $\astr$ is called \emph{(purely) periodic}.

\begin{lemma}\label{lem:periodic:subseq}
  Every arithmetic subsequence of an ultimately periodic sequence
  is ultimately periodic.
\end{lemma}
\begin{proof}
  Let $\astr\in\str{\aalph}$, $n_0\in\nat$, $t\in\posnat$
  with $\myall{n \ge n_0}{\astr(n+t) = \astr(n)}$.
  Further let $a\in\nat$, $b\in\posnat$. 
  We show that $\asub{\astr}{a}{b}$ is ultimately periodic.
  Let $n'_0\in\nat$ be such that $a + b n'_0 \ge n_0$,
  and let $t' \in \posnat$ with $bt' \equiv 0 \pmod t$ 
  (such $t'$ with $t' \le t$ always exists).
  Then we have 
  $\asub{\astr}{a}{b}(n+t') = \astr(a + bn + bt') = \astr(a + bn) = \asub{\astr}{a}{b}(n)$
  for all integers $n\ge n'_0$.
\end{proof}



We use $\primes = \{2,3,5,\ldots\}$ for the set of prime numbers.
We write 
$\residues{a}{k}$ for the congruence (or residue) class of $a$ modulo~$k$:
$\residues{a}{k} = \{ n \in \nat \where n \equiv a \pmod{k} \}$.
The set of residue classes modulo $k$ is denoted by $\congrclasses{k}$.
Moreover, for the union of 
$\residues{a_1}{k},\ldots,\residues{a_q}{k} \in \congrclasses{k}$
we write 
\[ 
  \residues{a_1,\ldots,a_q}{k}
  =
  \residues{a_1}{k} \cup \cdots \cup \residues{a_q}{k}
\]
We let $\mgroup{k} = (\congrclasses{k})^{\times} = \{ \, \residues{h}{k} \where \gcd{h}{k} = 1 \}$,
the multiplicative group of integers modulo~$k$.
%
A \emph{primitive root modulo~$k$} is a generator of the group $\mgroup{k}$.
That is, a primitive root is an element~$g$ of $\mgroup{k}$ 
such that for all $a\in\mgroup{k}$ 
there is some integer~$e$ with  $a \equiv g^e \pmod k$.
For integers~$e_1,e_2$ we have that 
$g^{e_1} \equiv g^{e_2} \pmod{k}$ if and only if $e_1 \equiv e_2 \pmod{k}$,
and so we have a bijection $e \mapsto g^e$ on $\mgroup{k}$, 
called \emph{discrete exponentation (with base~$g$)}.
Its inverse, $a \mapsto \log_g(a) = e$ with $e$ such that $a \equiv g^e \pmod{k}$,
is called \emph{discrete logarithm (to the base~$g$)}.

Groups 
\( 
  \mgroup{p} 
  = \{\,
      \residues{1}{p}\,,
      \residues{2}{p}\,,
      \ldots,
      \residues{p-1}{p}
    \}
  \simeq \{ 1, 2 , \dots , p-1 \}
\)
with $p$ an odd prime number contain all nonzero residue classes modulo~$p$, 
always have a primitive root and are cyclic.
For example the primitive roots modulo~$7$ are $3$ and $5$.
The powers of $3$ modulo $7$ are $3, 2, 6, 4, 5, 1, 3 \ldots$\,, 
and of $5$ they are $5, 4, 6, 2, 3, 1, 5, \ldots$, 
listing every number modulo~$7$ (except~$0$). 
On the other hand, $2$ is not a primitive root modulo~$7$: 
the powers of~$2$ modulo~$7$ are $2, 4, 1, 2, \ldots$,
missing several values from $\{1,\ldots,6\}$.
Note that a primitive root is not necessarily a prime number,
e.g., $6$ is a primitive root modulo~$11$.

\begin{lemma}\label{lem:discrete}
  \mbox{}
  \begin{enumerate}
    \item (Base change)\label{item:lem:discrete:base:change}
      Let $g$, $h$ be primitive roots mod $n$, $a \in \mgroup{n}$.
      Then $\log_{g}(a) = \log_{g}(h) \cdot \log_{h}(a)$.
    \item (Fermat's little theorem)\label{item:lem:discrete:fermat}
      Let $p$ be a prime number and $a\in\mgroup{p}$. 
      Then $a^{p-1} \equiv 1 \pmod p$.
    \item\label{item:lem:discrete:exp:-1}
      Let $p > 2$ be prime and $g$ be a primitive root modulo~$p$. Then
      $g^{(p-1)/2} \equiv -1 \pmod p$.
  \end{enumerate}
\end{lemma}
\begin{proof}
  Folklore.
  We only prove item~\eqref{item:lem:discrete:exp:-1}:
  by $g$ being a primitive root we know  $g^k \equiv -1 \pmod{p}$ 
  for some unique $1 \le k \le p-1$,
  and so $g^{2k} \equiv 1 \pmod{p}$ which in turn implies 
  $2k \equiv p-1 \pmod{p}$ by item~\eqref{item:lem:discrete:fermat}.
\end{proof}

\subsection*{Arithmetic Self-Similarity}

Throughout the paper we fix an arbitrary finite cyclic group
$\triple{\aalph}{{+}}{0}$.

We define the `arithmetic self-similarity' of an infinite sequence $\astr$, 
as the set of pairs $\pair{a}{b}$ 
such that the arithmetic subsequence $\asub{\astr}{a}{b}$ 
is similar to $\astr$, 
where we interpret `similar to' as {`is a vertical shift~of'}. 
\begin{definition}\label{def:AS}
  Let the relation $\ssimilar \subseteq \str{\aalph} \times \str{\aalph}$
  be defined by
  \begin{align*}
    \astr \similarx \bstr \quad\text{if and only if}\quad \astr = \bstr + c \text{ for some $c \in \aalph$}
    && (\astr,\bstr \in \str{\aalph})
  \end{align*} 
  Then, we define the \emph{arithmetic self-similarity} of a sequence $\astr\in\str{\aalph}$ by
  \begin{align*}
    \AS{\astr} = 
    \displset{\pair{a}{b}}{\text{$\asub{\astr}{a}{b} \similarx \astr$}}
  \end{align*}
\end{definition}
Clearly the relation $\ssimilar$ is an equivalence relation on $\str{\aalph}$
as the elements of the group $\aalph$ are invertible.
Equivalent sequences have the same arithmetic self-similarity.

\begin{lemma}\label{lem:AS:nw}
  For all $\astr,\bstr\in\str{\aalph}$,
  if $\astr \similarx \bstr$ 
  then $\AS{\astr} = \AS{\bstr}$.
  \qed
\end{lemma}

Arithmetic self-similarity is closed under composition of aritmetic progressions.
\begin{lemma}\label{lem:AS:closed:under:compostion}
  Let $\astr\in\str{\aalph}$ be an infinite sequence 
  and assume
  $\pair{a}{b},\pair{c}{d}\in\AS{\astr}$.
  Then also
  $\pair{a+bc}{bd}\in\AS{\astr}$.%
  \footnote{%
    For sequences $\astr\in\strp{\aalph}$, indexed by the positive integers,
    the statement has to be reformulated: 
    for all $a,b,c,d \in \posnat$,
    if $\pair{a}{b},\pair{c}{d} \in \AS{\astr}$, 
    then $\pair{a+b(c-1)}{bd} \in \AS{\astr}$.%
  }
  \qed
\end{lemma}

If for some $a \ge 1$ the shift $\asub{\astr}{a}{1}$ of a sequence $\astr\in\str{\aalph}$ 
is similar to $\astr$, then $\astr$ is periodic.
\begin{lemma}\label{lem:suffix:in:AS:periodic}
  Let $\astr\in\str{\aalph}$ and $a \ge 1$.  
  If $\pair{a}{1} \in \AS{\astr}$, then $\astr$ is periodic.%
  \footnote{%
    For sequences $\astr\in\strp{\aalph}$ this reads: 
    if $a \ge 2$ and $\pair{a}{1} \in \AS{\astr}$, then $\astr$ is periodic.%
  }
  \qed
\end{lemma}

\section{Completely Additive Sequences}\label{sec:additive}
In this section we investigate the arithmetic self-similarity of completely additive sequences.

\begin{definition}\label{def:additive}
  A sequence $\astr\in\strp{\aalph}$ is 
  \emph{completely additive (with respect to~$\aalph$)} 
  if it is a homomorphism from 
  the multiplicative monoid $\triple{\posnat}{{\cdot}}{1}$ of positive integers
  to the additive group $\triple{\aalph}{{+}}{0}$, that is
  \begin{align*}
    \nth{\astr}{1} = 0
    &&
    \nth{\astr}{nm} = \nth{\astr}{n} + \nth{\astr}{m}
    \quad(n,m\in\posnat)
    \label{eq:additive}
  \end{align*}
  The constant zero sequence~$\mbs{z}\in\strp{\aalph}$,
  defined by $\mbs{z}(n) = 0_{\aalph}$ ($n\in\posnat$),
  is called \emph{trivially additive}, or just \emph{trivial}.
\end{definition}

\begin{lemma}\label{lem:additive:periodic}
  Every sequence that is both completely additive and ultimately periodic 
  is trivial.
\end{lemma}
\begin{proof}
  Let $\astr\in\strp{\aalph}$ be a completely additive sequence. 
  First we assume $\astr$ to be purely periodic with period $t\in\posnat$.
  Then for all $n\in\posnat$ we have 
  $\nth{\astr}{nt} = \nth{\astr}{n} + \nth{\astr}{t}$ by complete additivity
  and $\nth{\astr}{nt} = \nth{\astr}{t}$ by $t$\nb-periodicity, whence $\nth{\astr}{n} = 0$.  
  
  Now let $\astr$ be ultimately periodic 
  with period~$t\in\posnat$ and preperiod~$n_0 \in \posnat$,
  i.e., $\astr(n+t) = \astr(n)$ for all integers $n \ge n_0$.
  It is clear that the arithmetic subsequence 
  $\div{\astr}{n_0} = \astr(n_0) \astr(2n_0)\ldots$  
  is periodic (see Lemma~\ref{lem:periodic:subseq}).
  By complete additivity we know $\astr(n n_0) = \astr(n) + \astr(n_0)$,
  and so $\astr$ is periodic as well. 
  Then, by the observation above, 
  we conclude that $\astr$ is the constant zero sequence.
\end{proof}

Completely additive sequences are uniquely determined by their values at prime number positions, 
i.e., by a map \mbox{$\ampy \funin \primes \to \aalph$}, as follows.
First we recall the definition of the $p$-adic valuation of an integer $n \ge 1$,
that is, the multiplicity 
of prime $p$ in the prime factorization of $n$:
\begin{definition}\label{def:p-adic:valuation}
  Let $p$ be a prime number. 
  The \emph{$p$-adic valuation} of $n \in \posnat$ 
  is defined by
  \begin{align*}
    \padval{p}{n} & = \maxx{\displset{a\in\nat}{\text{$p^a$ divides $n$}}}
  \end{align*}
\end{definition}
The infinite sequence $\pgs{p} = \padval{p}{1} \, \padval{p}{2} \,\ldots$ 
is completely additive (with respect to the cyclic group $\mbb{Z}$).

\begin{definition}\label{def:primgenseq}
  Let $\ampy \funin \primes \to \aalph$ and 
  define the sequence~$\pgs{\ampy}\in\strp{\aalph}$ for all $n\in\posnat$ by
  \begin{align*}
    \nth{\pgs{\ampy}}{n} 
    & = 
    \summ{p\in\primes}{}{\gtimes{\padval{p}{n}}{\funap{\ampy}{p}}}
  \end{align*}
  (Here summation is the generalized form of addition of the group~$\aalph$.)
  If the set $\{ p \in \primes \where \funap{\ampy}{p} \neq 0 \}$ is finite,
  then $\pgs{\ampy}$ is called \emph{finitely prime generated}. 
  Otherwise $\pgs{\ampy}$ is called an \emph{infinitely} prime generated sequence.
\end{definition}
Every completely additive sequence $\astr$ has a generator $\ampy \funin \primes \to \aalph$,
viz.\ $\ampy = \domrestr{\astr}{\primes}$, 
the domain restriction of the function $\astr \funin \posnat \to \aalph$ 
to the set of primes.
Hence, for every finite cyclic group~$\aalph$ with $\card{\aalph}\geq 2$,
the set of completely additive sequences over $\aalph$ has the cardinality of the continuum.
\begin{example}
  Define $\ampy \funin \primes\to\nalph{4}$ by
  $\funap{\ampy}{p} = 0,1,3,2,1$ if $p \equiv 1,2,3,4,0 \pmod{5}$, respectively.
  Then $\pgs{\ampy} = 01321 \, 01322 \, 01320 \, 01323 \, 01322 \ldots$ 
  is generated by the infinite set 
  $\{ p \where \funap{\ampy}{p} \neq 0 \} = \primes \cap \residues{2,3,4}{5} \cup \{5\}$.
\end{example}
If $\aalph = \nalph{2} = \{0,1\}$ 
then $\ampy \funin \primes \to \aalph_2$ 
is the characteristic function of a set $X\subseteq\primes$,
and we simply write $\funap{\pgs{X}}{n}$
($\mathrel{=} \summ{p\in X}{}{\padval{p}{n}} \mod 2$).
We sometimes write $\pgs{X}$ where $X\subseteq \nat$ 
to denote $\pgs{\primes\cap X}$.
For instance $\pgs{\residues{a}{b}}$
denotes the completely additive sequence generated by the set of prime numbers congruent to $a$ modulo~$b$,
which is infinite if $\gcd{a}{b} = 1$ by~\cite{diri:1837}.

\begin{example}
  The sequence $\terdragon = \funnap{\zeta}{\omega}{0}$ 
  with $\zeta\funin\setop{\wrd{\bit}}$ the morphism given by
  $\funap{\zeta}{0} = 010$
  and $\funap{\zeta}{1} = 011$,
  is the sequence of turns of the \textit{Terdragon curve}~\cite{davi:knut:1970} (\oeis{A080846}).
  From Theorem~\ref{thm:superthm} (and Proposition~\ref{prop:one:gap:morphism})
  it follows that $\terdragon$ is a completely additive sequence,
  generated by the infinite set of primes 
  congruent to~$2$ modulo~$3$, i.e., 
  \( \terdragon = \pgs{\residues{2}{3}} \). 
\end{example}

\begin{lemma}\label{lem:pgs:add}
  Let $\iampy{1},\iampy{2}\funin\primes\to\aalph_{2}$. 
  Then $\pgs{\iampy{1}+\iampy{2}} = \pgs{\iampy{1}} + \pgs{\iampy{2}}$.
  \qed
\end{lemma}
For sets $A,B\subseteq\primes$ Lemma~\ref{lem:pgs:add} says
$\pgs{A \setxorx B} = \pgs{A} + \pgs{B}$,
where $\ssetxor$ is symmetric difference:
$A \setxorx B = (A \cup B)\setminus(A \cap B) = (A \setminus B) \cup (B \setminus A)$,
and so we have
$\pgs{A \cup B} = \pgs{A} + \pgs{B}$ if $A \cap B = \emptyset$.
See for example Table~\ref{ref:table:pgs:2,3},
where the $2$ and $3$-adic valuation sequences in base $2$ are added modulo $2$.
\begin{table}[t]
\begin{footnotesize}
  \(
  \begin{array}{|c|c|c|c|c|c|c|c|c|c|c|c|c|c|c|c|c|c|c|c}
    \hline
          n   & 1 & 2 & 3 & 4 & 5 & 6 & 7 & 8 & 9 & 10 & 11 & 12 & 13 & 14 & 15 & 16 & 17 & 18 & \ldots \\
    \hline
    \pgs{2}   & 0 & 1 & 0 & 0 & 0 & 1 & 0 & 1 & 0 &  1 &  0 &  0 &  0 &  1 &  0 &  0 &  0 &  1 & \ldots \\
    \hline
    \pgs{3}   & 0 & 0 & 1 & 0 & 0 & 1 & 0 & 0 & 0 &  0 &  0 &  1 &  0 &  0 &  1 &  0 &  0 &  0 & \ldots \\
    \hline
    \pgs{2,3} & 0 & 1 & 1 & 0 & 0 & 0 & 0 & 1 & 0 &  1 &  0 &  1 &  0 &  1 &  1 &  0 &  0 &  1 & \ldots \\
    \hline
  \end{array}
  \)
  \end{footnotesize}%
  \caption{An instance of Lemma~\ref{lem:pgs:add}: $\pgs{2,3} = \pgs{2} + \pgs{3}$.}
  \label{ref:table:pgs:2,3}
\end{table}

{%
\renewcommand{\astr}{w}%

If $\astr\in\strp{\aalph}$ is additive,
it is clear that $\pair{b}{b} \in \AS{\astr}$ for all $b \in \posnat$. 
The question arises whether the arithmetic self-similarity of $\astr$
possibly contains other progressions.
This is not the case,
a result due to Kevin~Hare and Michael~Coons~\cite{hare:coon:2011}.

\begin{theorem}\label{thm:AS}
  Let $\astr\in\strp{\aalph}$ be a non-trivial completely additive sequence.
  Then: \[ \AS{\astr} = \displset{\pair{b}{b}}{b\in\posnat} \]
\end{theorem}

\begin{proof}[Proof~{(K.\:Hare, M.\:Coons~\cite{hare:coon:2011})}]
  \mbox{}
  \begin{itemize}

    \item[($\supseteq$)]
      Let $\astr$ be a completely additive sequence over~$\aalph$,
      and let $b\in\posnat$.
      For all $n \in \posnat$ we have
      \( \asub{\astr}{b}{b}(n) = 
         \astr(bn) = \astr(n) + \astr(b) \).
         Hence $\asub{\astr}{b}{b} \similarx \astr$ 
         and so $\pair{b}{b}\in\AS{\astr}$.

    \item[($\subseteq$)]
      Let $\astr$ be a completely additive sequence over~$\aalph$
      and assume $\pair{a}{b} \in \AS{\astr}$ with $a \ne b$.
      We will show that $\astr$ is trivially additive,
      i.e., $\nth{\astr}{n} = 0$ for all $n \in \posnat$.
      
      \begin{claim}\label{claim}
        Without loss of generality we may assume
        $\abs{a-b} = \pm 1$.
      \end{claim}
      \begin{proof}[Proof of Claim~\ref{claim}]
        By the assumption $\pair{a}{b} \in \AS{\astr}$ we have 
        \begin{align*}
          \nth{\asub{\astr}{a}{b}}{n} 
          = \nth{\astr}{a + b(n-1)} 
          = \nth{\astr}{bn + a - b} 
          = \nth{\astr}{n} + c
          && (n \in \posnat)
        \end{align*}
        for some $c \in \aalph$.
        So in particular it holds for $n = \abs{a-b} \cdot N$, 
        where $N$ is some arbitrary positive integer.
        This implies that
        \begin{align*}
          \nth{\astr}{b \cdot \abs{a-b} \cdot N + a - b} 
          = \nth{\astr}{\abs{a-b}} + \nth{\astr}{N} + c
        \end{align*}
        but
        \begin{align*}
          \nth{\astr}{b \cdot \abs{a-b} \cdot N + a - b} 
          = \nth{\astr}{\abs{a-b}} + \nth{\astr}{bN \pm 1)}
        \end{align*}
        Subtracting $\nth{\astr}{\abs{a-b}}$ from both right-hand sides, 
        we have
        \begin{align*}
          \nth{\astr}{bN + \sign{a-b}} = \nth{\astr}{N} + c
          && (a > b)
        \end{align*}
        which concludes the proof of Claim~\ref{claim}.
      \end{proof}
      We employ the techniques used in the proof of Theorem~4 
      in~\cite{cass:fere:maud:riva:sark:2000} for showing that 
      in both cases $\astr$ is trivially additive.
      \begin{case}\label{case:1}
        If there is some $c\in\aalph$ such that 
        $\nth{\astr}{bn + 1} = \nth{\astr}{n} + c$ for all $n \in \posnat$, 
        then $\astr$ is trivial.
      \end{case}
      \begin{proof}[Proof of Case~\ref{case:1}]
        Let $c\in\aalph$ be such that $\nth{\astr}{bn + 1} = \nth{\astr}{n} + c$
        for all $n$.
        We will show by induction that 
        \begin{align}
          \nth{\astr}{bn + i} = \nth{\astr}{n} + c
          && (n \in \posnat)
          \tag{$\ddagger$}
          \label{eq:case:1}
        \end{align}
        for all $i = 1,2,\ldots$.
        Hence, taking $n = 1$, 
        the sequence $\astr$ is constant from position $b + 1$ onward.
        With Lemma~\ref{lem:additive:periodic} it then follows that $\astr$ is trivially additive.
        
        For $i = 1$ equation \eqref{eq:case:1} holds by assumption.
        Now assume \eqref{eq:case:1} holds for $i = 1,2,\ldots,j$.
        Then we find:
        \begin{align*}
          \nth{\astr}{bn + j} + \nth{\astr}{bn + 1}
          & = 2\cdot\nth{\astr}{n} + 2c
        \end{align*}
        and 
        \begin{align*}
          \nth{\astr}{bn + j} + \nth{\astr}{bn + 1}
          & = \nth{\astr}{(bn + j)(bn + 1)} \\
          & = \nth{\astr}{b^2 n^2 + bn(j+1) + j} \\
          & = \nth{\astr}{b(bn^2 + n(j+1)) + j} \\
          & = \nth{\astr}{bn^2 + n(j+1)} + c \\
          & = \nth{\astr}{(bn + j + 1)n} + c \\
          & = \nth{\astr}{bn + j + 1} + \nth{\astr}{n} + c
        \end{align*}
        Hence $\nth{\astr}{bn + j + 1} = \nth{\astr}{n} + c$
        and the result follows by induction. 
        This concludes the proof of Case~\ref{case:1}.
      \end{proof}
      
      \begin{case}\label{case:2}
        If there is some $c\in\aalph$ such that 
        $\nth{\astr}{bn + 1} = \nth{\astr}{n} + c$ for all $n \in \posnat$, 
        then $\astr$ is trivial.
      \end{case}
      \begin{proof}
        Assume, for some $c\in\aalph$ that 
        $\nth{\astr}{bn - 1)} = \nth{\astr}{n} + c$ for all integers $n \in \posnat$.
        Then
        \begin{align*}
          \nth{\astr}{b(bn^2) - 1)} 
          & = \nth{\astr}{(bn - 1)(bn + 1)}  \\
          & = \nth{\astr}{bn - 1} + \nth{\astr}{bn + 1} \\
          & = \nth{\astr}{n} + c + \nth{\astr}{bn + 1} 
        \end{align*}
        and 
        \begin{align*}
          \nth{\astr}{b(bn^2) - 1)} 
          &= \nth{\astr}{bn^2} + c \\
          & = \nth{\astr}{b} + 2\cdot\nth{\astr}{n} + c
        \end{align*}
        Hence $\nth{\astr}{bn + 1} = \nth{\astr}{n} + \nth{\astr}{b}$, for all $n \in \posnat$.
        By Case~\ref{case:1} it then follows that $\astr$ is trivially additive.
      \end{proof}

  \end{itemize}
  To establish the direction~$\subseteq$,
  we have shown that 
  if $\astr$ is completely additive
  and $\pair{a}{b} \in \AS{\astr}$ 
  for some $a \ne b$, then $\astr$ is trivially additive.
\end{proof}%
}%

\begin{remark}\label{rem:hare}
  We were pleasantly surprised to receive an e\tsp-mail from Kevin Hare
  with the beautiful proof of Theorem~\ref{thm:AS}~\cite{hare:coon:2011}.
  At the time we had some partial results:
  we had a proof of the statement for finitely prime generated sequences~$\astr$,
  and also for some specific infinitely prime generated sequences, 
  namely those produced by one-gap Toeplitz patterns 
  (see Sections~\ref{sec:toeplitz} and~\ref{sec:toeplitz:additive}).
  For these partial results we refer to the first version of the current arXiv report, 
  available at the following url: 
  \url{http://arxiv.org/pdf/1201.3786v1}.
  Finally, we had also established the statement 
  for completely additive sequences over the infinite cyclic group~$\mbb{Z}$.
\end{remark}

\section{Toeplitz Words}\label{sec:toeplitz}
Toeplitz words were introduced in~\cite{jaco:kean:1969}, see also 
\cite{allo:bach:1992,cass:karh:1997}.
A Toeplitz word over an alphabet~$\aalph$ is an infinite sequence 
iteratively constructed as follows: 
Given is a starting sequence $\iastr{0}$ over the alphabet~$\aalph\cup\{\gap\}$,
where we may think of the symbol~$\gap\not\in\aalph$ as `undefined', 
and of $\iastr{0}$ as a `partially defined' sequence.
The occurrences of~$\gap$ in a sequence~$\astr$ are called the `gaps' of~$\astr$.
For $i > 0$ 
the sequence~$\iastr{i}$ is obtained from $\iastr{i-1}$
by `filling its sequence of gaps' 
(i.e., consecutively replacing the occurrences of~$\gap$ in $\iastr{i-1}$) 
by the sequence $\iastr{i-1}$ itself,
as made precise in Definition~\ref{def:toeplitz:inf} below.
In the limit we then obtain a totally defined sequence 
(i.e., without gaps) if and only if the first symbol of 
the start sequence~$\iastr{0}$ is defined (i.e., is in~$\aalph$).


As in~\cite{allo:bach:1992} we allow the application 
of any bijective map~$\agap\funin\aalph\to\aalph$ to the symbols that replace the gaps. 
Thus we let the above $\iastr{i}$ be sequences over $\aalph \cup \symgroup{\aalph}$ 
where $\symgroup{\aalph}$ denotes the symmetric group of bijections (or permutations) on $\aalph$,
and let the elements from $\symgroup{\aalph}$ play the role of gaps.
Filling a gap~$\agap$ with a letter~$a\in\aalph$ 
then results in $\funap{f}{a}$, and filling a gap~$\agap$ by a gap~$\bgap$ 
results in the function composition $\fgap{f \circ g}$.
Viewed in this way the symbol $\gap$ stands for $\mrm{id}_\aalph$, 
the identity element of~$\symgroup{\aalph}$, and we will use it in that way.
From the finiteness of the alphabet~$\aalph$ and the $f$'s being one-to-one 
it directly follows that extending the set of Toeplitz patterns 
does \emph{not} increase the expressive power of the system~\cite{allo:bach:1992}: 
every Toeplitz word generated by a pattern with gaps from $\symgroup{\aalph}$ 
can already be defined using a (longer) pattern with gaps~$\gap$ only.
In other words, it is a conservative extension.
\begin{definition}\label{def:toeplitz:inf}
  For infinite words $\astr ,\bstr\in\strp{(\gaalph)}$
  we define $\Tsubst{\astr}{\bstr}$ recursively by
  \begin{align*}
    \Tsubst{(a \, x)}{y} & = a \, \Tsubst{x}{y} &
    \Tsubst{(\agap \, x)}{b \, y} & = \funap{\agap}{b} \, \Tsubst{x}{y} &
    \Tsubst{(\agap \, x)}{\bgap \, y} & = (\fgap{\funcomp{\agap}{\bgap}}) \, \Tsubst{x}{y}
  \end{align*}
  where $a,b\in\aalph$, $\agap,\bgap\in\symgroup{\aalph}$,
  and $x,y\in\strp{(\gaalph)}$.
  Further let $\apat \in \aalph\wrd{(\gaalph)}$ (first symbol not a gap)
  and, for $k=0,1,2,\ldots$, define $\toeplitzi{k}{\apat}\in\strp{(\gaalph)}$ by
  \begin{align*}
    \toeplitzi{0}{\apat}   & = \gap^\omega = \gap\gap\gap\ldots &
    \toeplitzi{k+1}{\apat} & = \Tsubst{\apat^\omega}{\toeplitzi{k}{\apat}}
  \end{align*}
  Then $\toeplitz{\apat} \in \strp{\aalph}$, 
  the \emph{Toeplitz word generated by $\apat$}, is defined as the limit of these words,
  as follows:
  \[
    \toeplitz{\apat} = \lim_{k\to\infinity}{\toeplitzi{k}{\apat}}
  \]
  We let $\gapnr{\apat} = \length{\{h\where\nth{\apat}{h}\in\symgroup{\aalph}\}}$        
  denote \emph{the number of gaps in $\apat$},
  and, following~\cite{cass:karh:1997}, 
  we call $\toeplitz{\apat}$ a Toeplitz word \emph{of type $\pair{r}{q}$}
  when $r = \length{\apat}$ and $q = \gapnr{\apat}$. 
\end{definition}%

\begin{example}
  Let $\apat = 0\gap 1\gap \in \wrd{(\nalph{2}\cup\symgroup{\nalph{2}})}$.
  The sequence of `$\apat$-generations' $\toeplitzi{0}{P},\toeplitzi{1}{P},\ldots$ 
  starts as follows:
  \begin{align*}
    \gap^\omega ,\,
    (0\gap 1\gap)^\omega ,\,
    (001\gap 011\gap)^\omega ,\,
    (0010 011\gap 0011 011\gap)^\omega ,\,
    \ldots
  \end{align*}
  The limit of this sequence of sequences 
  is the well-known regular paperfolding sequence: 
  $\toeplitz{\apat} = \PF$, 
  see~\cite{allo:bach:1992}; 
  $\PF$ is additive and 
  generated by the (infinite) set of prime numbers congruent to $3$ modulo~$4$.
\end{example}

\begin{proposition}
  The set $\strp{(\aalph \cup \symgroup{\aalph})}$ 
  with the operation $\pair{\astr}{\bstr} \mapsto \Tsubst{\astr}{\bstr}$ 
  and identity element~$\gap^\omega$ forms a monoid.
\end{proposition}

The Toeplitz word 
$\toeplitz{\apat}$ is the unique solution of $x$ in the equation
\( x = \Tsubst{\apat^\omega}{x} \).
The construction can thus be viewed as `self-reading' in the sense that 
the sequence under construction itself is read to fill the gaps.
For example, the period doubling sequence~$\PD$ is the Toeplitz word over $\nalph{2}$
generated by the pattern $010\gap$, as follows:
\newcommand{\gp}{\mbs{\gap}\hspace{0.035em}}
\begin{align*}
  (010\gap)^\omega
  & =
  010\gp 010\gp 010\gp 010\gp 
  010\gp 010\gp 010\gp 010\gp 
  010\gp 010\gp 010\gp 010\gp 
  \ldots \\
  \PD = \toeplitz{010\gap}
& =
  010\mbf{0} 010\mbf{1} 010\mbf{0} 010\mbf{0} 
  010\mbf{0} 010\mbf{1} 010\mbf{0} 010\mbf{1}
  010\mbf{0} 010\mbf{1} 010\mbf{0} 010\mbf{0} 
  \ldots
\end{align*}%
For $d\in\aalph$ let us denote by $\rotgap{d}$ the rotation $a \mapsto a+_\aalph d$ (so $\gap = \rotgap{0}$).
Then the pattern $010\gap$ can be simplified to $0\rotgap{1}$, 
because 
$0\rotgap{1} \toeplitzeq 010\gap$, 
where the relation~$\stoeplitzeq$ is defined by:
\[ \apat \toeplitzeq \bpat \quad \text{if and only if} \quad \toeplitz{\apat} = \toeplitz{\bpat} \]
We call the elements of the set $\newrd{(\aalph \cup \symgroup{\aalph})}$
\emph{\hyph{$\aalph$}{patterns}}. 
We continue with a definition of a 
composition operation~$\sTcmp$ 
on $\aalph$-patterns~$\apat,\bpat$ such that 
$(\Tcomp{\apat}{\bpat})^\omega = \Tsubst{\apat^\omega}{\bpat^\omega}$.
This then explains the equivalence $0\rotgap{1} \toeplitzeq 010\gap$ 
since $(0\rotgap{1}) \Tcmpx (0\rotgap{1}) = 010\gap$
and because the equivalence classes induced by $\stoeplitzeq$ are closed under composition.
The 
idea of composing patterns $P,Q$ is to first take copies $P^n$ and $Q^m$ 
such that $\gapnr{P^n} = \length{Q^m}$ and then fill the sequence of gaps through $P^n$ by $Q^m$.
Recall that we defined $\gscale{n}{m} = \frac{m}{\gcd{n}{m}}$
(so $n \cdot \gscale{n}{m} = m \cdot \gscale{m}{n}$).
\begin{definition}\label{def:toeplitz}
  Let $\apat,\bpat$ be \hyph{$\aalph$}{patterns}, and define their \emph{Toeplitz composition} as follows:
  \begin{align*}
    \Tcomp{\apat}{\bpat} & = 
    \begin{cases}
    \apat & \text{if $\gapnr{\apat} = 0$} \\
    \Tsubst{\apat^{d_1}}{\bpat^{d_2}} 
    & \text{if $\gapnr{\apat} \gt 0$, $d_1 = {\gscale{\gapnr{P}}{\length{Q}}}$, and $d_2 = {\gscale{\length{Q}}{\gapnr{P}}}$.}
    \end{cases}  
  \end{align*}
  where $\Tsubst{u}{v}$ is defined for all words $u,v\in\wrd{(\gaalph)}$ 
  such that $\gapnr{u} = \length{v}$ by 
  \begin{align*}
    \Tsubst{x}{\wrdemp} & = x &
    \Tsubst{(a \, x)}{y} & = a \,  \Tsubst{x}{y} &
    \Tsubst{(\agap \, x)}{b \, y} & = \funap{\agap}{b} \,  \Tsubst{x}{y} &
    \Tsubst{(\agap \, x)}{\bgap \, y} & = (\fgap{\funcomp{\agap}{\bgap}}) \, \Tsubst{x}{y}
  \end{align*}
  where $a,b\in\aalph$, $\agap,\bgap\in\symgroup{\aalph}$, and $x,y \in \wrd{(\gaalph)}$.%
  \footnote{Note that the recursive calls of $\Tsubst{u}{v}$ preserve the requirement $\gapnr{u} = \length{v}$.}

  For a \hyph{$\aalph$}{pattern}~$\apat$ and integer $k\ge 0$ 
  we define $\Tpow{\apat}{k}$ 
  by $\Tpow{\apat}{0} = \gap$ and $\Tpow{\apat}{n+1} = \apat \Tcmpx \Tpow{\apat}{n}$.
\end{definition}
\begin{example}
  To compute 
  $(0\gap\gap 1\gap\gap) \Tcmpx (234567)$
  we take $d_1 = 3$ and $d_2 = 2$:
  \[
    (0\gap\gap 1\gap\gap) \Tcmpx (234567) 
    = \Tsubst{(0\gap\gap 1\gap\gap)^3}{(234567)^2}
    = 023 145 \, 067 123 \, 045 167
  \]
\end{example}

\begin{proposition}\label{prop:Tpat:monoid}
  The set of $\aalph$-patterns forms a monoid 
  with pattern composition as its operation
  and $\gap$ as its identity element. 
\end{proposition}

\begin{lemma}
  The map $\apat \mapsto \apat^\omega$ is a monoid homomorphism 
  from $\triple{\wrd{(\gaalph)}}{\sTcmp}{\gap}$
  to $\triple{\strp{(\gaalph)}}{\pair{\astr}{\bstr}\mapsto\Tsubst{\astr}{\bstr}}{\gap^\omega}$. 
\end{lemma}
This immediately implies $(\Tpow{\apat}{k})^\omega = \toeplitzi{k}{\apat}$, 
and hence $\toeplitz{\apat} = \lim_{k\to\infinity}{\Tpow{\apat}{k}}$.

The length and number of gaps of a composed pattern $\Tcomp{P}{Q}$ are computed as follows:
\begin{lemma}\label{lem:Tcmp:length}
  $\length{\apat \Tcmpx \bpat} = \gscale{\gapnr{\apat}}{\length{\bpat}} \cdot \length{\apat}$
  \,and\, 
  $\gapnr{\apat \Tcmpx \bpat} = \gscale{\length{\bpat}}{\gapnr{\apat}} \cdot \gapnr{\bpat}$ 
\end{lemma}

The congruence classes induced by $\stoeplitzeq$ are closed under 
concatenation and composition.
\begin{lemma}\label{lem:pattern:composition}
  Let $\apat$ be a $\aalph$-pattern and $k \geq 1$.
  Then $\apat^{k} 
  \toeplitzeq \apat$ 
  and $\apat^{(k)} 
  \toeplitzeq \apat$.
\end{lemma}

\begin{example}
%
  The classical Hanoi sequence~\oeis{A101607}
  is the sequence $\hanoi$ of moves 
  such that the prefix of length $2^N-1$ of $\hanoi$
  transfers $N$ disks 
  from peg $A$ to peg $B$ if $N$ is odd,
  and to peg~$C$ if $N$ is even. 
%
%
  We represent moving the topmost disk 
  from peg~$X$ to $Y$ by the pair~$\pair{X}{Y}$, 
  and map moves to $\nalph{6}$ using 
  \begin{align*}
    \pair{A}{B} & \mapsto 0 & \pair{B}{C} & \mapsto 2 & \pair{C}{A} & \mapsto 4 \\
    \pair{B}{A} & \mapsto 1 &
    \pair{C}{B} & \mapsto 5 & 
    \pair{A}{C} & \mapsto 3
  \end{align*}
  In~\cite{allo:dres:1990,allo:bach:1992} it is shown that 
  the sequence $\hanoi$ is the Toeplitz word generated by the pattern $032\gap450\gap214\gap$\,:
  \begin{align*} 
  \hanoi 
  & = \toeplitz{032\gap450\gap214\gap} \\
  & = 
    032 0 450 3 214 2 
    032 0 450 4 214 5 
    032 0 450 3 214 2 
    032 1 450 4 214 2 
    \ldots 
  \end{align*}
  Now let $f = \rotgap{3}$,
  the involution on $\nalph{6}$ 
  corresponding to swapping $B$s and $C$s in moves~$\pair{X}{Y}$.
  Then the above pattern can be simplified to $0f2f4f$, that is,
  \begin{align*}
    \hanoi = \toeplitz{0f2f4f}
  \end{align*}
  as composing the pattern~$0f2f4f$ with itself yields the pattern ${032\gap450\gap214\gap}$
  from above:
  \begin{align*}
       \Tcomp{(0f2f4f)}{(0f2f4f)}
       = {032f^2 450f^2 214 f^2} = {032\gap450\gap214\gap}
  \end{align*}
  The sequence of directions obtained by taking $\hanoi$ modulo $2$
  (we took even (odd) numbers to represent (counter)clockwise moves) 
  is the period doubling sequence~$\PD$, see~\cite{allo:bach:1992}:
  \begin{align*}
    (\hanoi \bmod 2) = (\toeplitz{0f2f4f} \bmod 2) = \toeplitz{0\rotgap{1}0\rotgap{1}0\rotgap{1}} = \toeplitz{0\rotgap{1}} = \PD 
  \end{align*}

\end{example}

The recurrence equations for a Toeplitz word are easy to establish.
\begin{lemma}\label{lem:toeplitz:recurrence}
  Let $P = a_1 \ldots a_r$ be a $\aalph$-pattern with $a_1 \in \aalph$, $r \ge 2$. 
  Let $h_1 < h_2 < \ldots < h_q$ be the sequence of indices $h$ 
  such that $a_h\in\symgroup{\aalph}$ (so $P$ has $q \lt r$ gaps). 
  Then for all $n \in \nat$:
  \vspace{-1ex}
  \begin{align*}
    \nth{\toeplitz{P}}{r n + i}   & = a_i 
    && \text{if $1 \le i \le r$ and $a_i \in \aalph$\,, and} \\
    \nth{\toeplitz{P}}{r n + h_j} & = \funap{\agap}{\nth{\toeplitz{P}}{q n + j}}
    && \text{for all $1 \le j \le q$ with $a_{h_j} = \agap$\punc.}
  \end{align*}
\end{lemma}
\begin{example}
  Consider the $\aalph$-pattern $\apat = a\agap b\bgap\cgap$ 
  for some $a,b\in\aalph$ and $\agap,\bgap,\cgap\in\symgroup{\aalph}$.
  With Lemma~\ref{lem:toeplitz:recurrence} we obtain the following 
  recurrence equations for $\astr = \toeplitz{\apat}$, for all $n\in\nat$:
  \begin{align*}
    \nth{\astr}{5n+1} & = a &
    \nth{\astr}{5n+2} & = \funap{\agap}{\nth{\astr}{3n+1}} & 
    \nth{\astr}{5n+3} & = b \\ 
    \nth{\astr}{5n+4} & = \funap{\bgap}{\nth{\astr}{3n+2}} & 
    \nth{\astr}{5n+5} & = \funap{\cgap}{\nth{\astr}{3n+3}}
  \end{align*}
\end{example}

Toeplitz words of type~$\pair{r}{1}$ can be obtained by 
iterating an $r$-uniform morphism~\cite{cass:karh:1997},
whence they are $r$-automatic~\cite{allo:shal:2003}.
\begin{proposition}\label{prop:one:gap:morphism}
  \mbox{}\nopagebreak
  \begin{enumerate}
    \item 
      Let $\apat\in\aalph\wrd{(\gaalph)}$ and define 
      $h \funin \setop{\wrd{\aalph}}$ 
      by $\funap{h}{a} = \Tcomp{\apat}{a}$ ($a\in\aalph$).
      Then $\funnap{h}{\omega}{a} = \toeplitz{\apat}$.
    \item
      Let $h \funin \setop{\wrd{\aalph}}$ be the morphism defined by
      $\funap{h}{a} = b u \funap{\agap}{a} v$ ($a \in \aalph$)
      for some fixed $b\in\aalph$, $u,v\in\wrd{\aalph}$, 
      and $\agap\in\symgroup{\aalph}$.
      Define the $\aalph$-pattern $P$ by $P = b u \agap v$.
      Then $\toeplitz{\apat} = \funnap{h}{\omega}{b}$.
  \end{enumerate}
\end{proposition}

\section{Additive Toeplitz Words \\ Generated by Single-Gap Patterns}\label{sec:toeplitz:additive}
We characterize additive Toeplitz words of type $\pair{r}{1}$. 
More precisely, we characterize the set~$X$ of \hyph{one}{gap} Toeplitz patterns~$\apat$ 
such that $\toeplitz{\apat}$ is additive if and only if $\apat$ is in $X$.
As it turns out, the discrete logarithm to the base $g$
modulo a prime number~$p$ (see Section~\ref{sec:prelims})
plays a key role in the construction of these patterns. 

We adopt the following convention:
Whenever a $\aalph$-pattern~$\apat$ generates a non-surjective Toeplitz word
(i.e., $\nth{\toeplitz{\apat}}{\posnat} \subsetneq \aalph$),
we identify gaps~$\agap$ occurring in $\apat$
with all bijections that coincide with $\agap$ on the letters occurring in $\toeplitz{\apat}$.

\begin{lemma}\label{lem:gap:end}
  Let $\apat = a_1 a_2 \ldots a_\ell$ 
  such that $\toeplitz{\apat}$ is non-trivially additive. 
  Then $a_1 = 0$ 
  and $a_\ell = \rotgap{d}$ for some $d\in\aalph$.
\end{lemma}
\begin{proof}
  Let $\astr = \toeplitz{\apat}$.
  Additive sequences have intial value $0$, so $\nth{\astr}{1} = a_1 = 0$. 
  Moreover if $a_\ell \in \nalph{k}$ 
  then it follows from the definition of Toeplitz words that $\div{\astr}{\ell}$ 
  is a constant sequence:
  $\nth{\astr}{n\ell} = a_\ell$ for all $n\in\posnat$.
  On the other hand, we also have 
  $\div{\astr}{\ell} \similarx \astr$ by complete additivity of~$\astr$, that is,
  $\nth{\astr}{n\ell} = \nth{\astr}{n} + \nth{\astr}{\ell}$ for all $n\in\posnat$.
  This combination of facts only occurs if $\astr$ is constant zero, 
  contradicting the assumption that $\astr$ is non-periodic.
  Finally, by additivity of $\astr$ it follows that
  $\div{\astr}{\ell} = \astr + \nth{\astr}{\ell}$ 
  and hence 
  the bijection at position $\ell$ has to be a rotation
  for all elements in the image $\astr(\posnat)$.
\end{proof}

\newcommand{\CAP}{\mfrk{P}}
\noindent
We let $\CAP$ denote the set of one-gap $\aalph$-patterns
generating additive sequences: 
\begin{align*}
  \CAP 
  = \displset{\apat\in\wrd{(\galph{\aalph})}}
             {\text{$\toeplitz{\apat}$ is additive and $\gapnr{\apat} = 1$}}
\end{align*}

\begin{lemma}\label{lem:pat:composite:constant}
  Let $\apat\in\CAP$ with $\length{\apat} = nm$ for some $n,m\geq 2$.
  Then the arithmetic subsequence $\asub{\toeplitz{\apat}}{i}{n}$ is constant
  for every $i$ with $1 \le i \lt n$.
\end{lemma}
\begin{proof}
  Let $\astr = \toeplitz{\apat}$. 
  By Lemma~\ref{lem:gap:end} we know that the (only) gap of $P$ is at the end. 
  Hence $a_i\in\aalph$ and from Lemma~\ref{lem:toeplitz:recurrence} it follows that 
  $\asub{\astr}{im}{nm} = \asub{(\asub{\astr}{m}{m})}{i}{n}$ is constant.
  Moreover $\asub{(\asub{\astr}{m}{m})}{i}{n} = \asub{(\astr + \nth{\astr}{m})}{i}{n} = \asub{\astr}{i}{n} + \nth{\astr}{m}$ 
  by additivity of $\astr$. Hence also $\asub{\astr}{i}{n}$ is constant.
\end{proof}

\begin{lemma}\label{lem:pat:composite:not:prime:power:constant:zero}
  Let $\apat\in\CAP$ with $\length{\apat}$ not a prime power.
  Then $\toeplitz{\apat}$ is the constant zero sequence.
\end{lemma}
\begin{proof}
  Let $\apat$ be non-trivial and $\length{\apat} = nm$ for some $n,m \gt 1$ with $\gcd{n}{m} = 1$
  and $\astr = \toeplitz{\apat}$.
  From Lemma~\ref{lem:pat:composite:constant} it follows that 
  all subsequences $\asub{\astr}{i}{n}$ and $\asub{\astr}{j}{m}$ are 
  constant for $1 \leq i < n$ and $1 \leq j < m$.
  Since $n$ and $m$ are relatively prime we have 
  that for every $j < m$ there is a $k \in \posnat$ 
  such that $kn + 1 \equiv j \pmod{m}$.
  Hence the subsequence $\asub{\astr}{1}{n}$ contains elements 
  of every subsequence $\asub{\astr}{j}{m}$ ($1 \leq j < m$),
  and thus $\asub{\astr}{1}{m} = \asub{\astr}{2}{m} = \ldots = \asub{\astr}{m-1}{m}$.
  Similarly we obtain $\asub{\astr}{1}{n} = \asub{\astr}{2}{n} = \ldots = \asub{\astr}{n-1}{n}$.
  Consequently we have $\nth{\astr}{k} = 0$ if $k \nequiv 0 \pmod{m}$ or $k \nequiv 0 \pmod{n}$.
  Hence $\nth{\astr}{k} = 0$ if $k \nequiv 0 \pmod{\lcm{n}{m}}$.
  From Lemma~\ref{lem:gap:end} we know that the only gap is at position $nm$,
  and since $\lcm{n}{m} = nm = \length{\apat}$ 
  we get $\nth{\apat}{k} = 0$ for all $1 \leq k < nm$.
  Finally the gap at $nm$ has to map $0$ to $0$ 
  because $\nth{\astr}{nm} = \nth{\astr}{n} + \nth{\astr}{m} = 0$.
\end{proof}

\begin{lemma}\label{lem:additive:pattern:decompose}
  Let $\apat = a_1 a_2 \ldots a_{r}\in\CAP$ 
  with $r = p^k$ for some prime~$p$ and~$k\ge2$,
  and $\toeplitz{\apat} \ne 0^\omega$. 
  Then 
  $\apat = \Tpow{\bpat}{k}$ 
  with $\bpat = a_1 a_2 \ldots a_{p-1} \rotgap{a_p} \in \CAP$.
  (Hence $\apat \toeplitzeq \bpat$ by Lemma~\ref{lem:pattern:composition}.)
\end{lemma}
\begin{proof}
  With Lemma~\ref{lem:gap:end} we have that $a_1, \ldots, a_{p^k-1} \in \aalph$ and $a_{p^k} \in \symgroup{\aalph}$ ($*$).
  From Lemma~\ref{lem:pat:composite:constant} we know that 
  all subsequences $\asub{\astr}{i}{p}$ with $1 \le i \lt p$ are constant.
  Hence, using ($*$), $\apat$ and $\Tpow{\bpat}{k}$ coincide on all positions $j \nequiv 0 \pmod{p}$.
  Also, by $\toeplitz{\apat}$ being additive and the definition of pattern composition (Definition~\ref{def:toeplitz}), 
  we have 
  $\nth{\apat}{mp} = a_m + a_p = \nth{\Tpow{\bpat}{k}}{mp}$ for every $1 \le m \lt p$. 
  Finally, $\nth{\apat}{p^k} = a_p + a_{p^{k-1}} = \ldots = k \cdot a_p = \nth{\Tpow{\bpat}{k}}{p^k}$.
\end{proof}

So far we have shown that all one-gap patterns which generate non-trivial additive Toeplitz words
have the gap at the end (Lemma~\ref{lem:gap:end}), of prime power length 
(Lemma~\ref{lem:pat:composite:not:prime:power:constant:zero}), 
and can always be decomposed to a pattern of prime length (Lemma~\ref{lem:additive:pattern:decompose}).
Next we give the exact shape of these atomic patterns, 
which are defined using discrete logarithms.
\begin{definition}
  Let $g$ be a primitive root of some prime $p>2$. 
  We define words $\dlog{p,g} \in \wrd{\nalph{p-1}}$ and $\dlog{2,1} \in \wrd{\nalph{1}}$ by
  \begin{align*}
    \dlog{p,g} = 0 \, \log_g(2) \, \log_g(3) \, \ldots \, \log_g(p-1)
    &&
    \dlog{2,1} = 0 
  \end{align*} 
\end{definition}

We show that every atomic pattern in $\CAP$ is of the form $(\gtimes{\dlog{p,g}}{c})\rotgap{d}$ 
for some $c,d\in\aalph$.

\begin{theorem}\label{thm:dlog:additive:sound}
  Let $\apat\in\CAP$ with $\length{\apat} = p$ for some prime~$p$,
  $\toeplitz{\apat}$ be non-trivial, and $g$ a primitive root modulo~$p$.
  Then $\apat = (\gtimes{ \dlog{p,g} }{ \nth{\apat}{g} }) \,\rotgap{d}$ for some $d \in \aalph$.
\end{theorem}
\begin{proof}
  By Lemma~\ref{lem:gap:end} we know that the only gap is at the end and that it is a rotation.  
  Let $1 \le i \lt p$. 
  By Lemma~\ref{lem:toeplitz:recurrence} and additivity of $\toeplitz{\apat}$
  we obtain
  $\nth{\toeplitz{\apat}}{g^{\nth{\dlog{p,g}}{i}}} = \nth{\dlog{p,g}}{i} \cdot \nth{\apat}{g}$.
\end{proof}

\begin{example}
  Let $\aalph = \nalph{4}$, and $\apat = 002022\rotgap{3}$. 
  This gives rise to the following completely additive sequence%
  \footnote{%
    Additivity of $\toeplitz{\apat}$ can be checked with Lemma~\ref{lem:cap:gap}.%
  }%
  :
  \[
    \toeplitz{\apat} =
    002022 3\,  002022 3\, 002022 1\, 002022 3\, 002022 1\, 002022 1\, 002022 2\, \ldots \,
  \]
  Theorem~\ref{thm:dlog:additive:sound} gives
  \begin{align*}
    \apat & = (\gtimes{\dlog{7,3}}{\apat(3)})\,\rotgap{3} = (\gtimes{021453}{2})\,\rotgap{3} \\
    \apat &  = (\gtimes{\dlog{7,5}}{\apat(5)})\,\rotgap{3} = (\gtimes{045213}{2})\,\rotgap{3}
  \end{align*}
\end{example}
The question remains whether every pattern of the form $\apat = (\gtimes{\dlog{p,g}}{c}) \rotgap{d}$
gives rise to an additive sequence $\toeplitz{\apat}$.
The following example shows this is not the case.
Theorem~\ref{thm:dlog:additive:complete} formulates the exact requirement for $c$\,
so that $\toeplitz{\apat}$ is additive.
\begin{example}
  Let $\aalph = \nalph{4}$, and 
  $\bpat = (\gtimes{\dlog{7,3}}{3})\,\rotgap{3} = (\gtimes{021453}{3}) \,\rotgap{3} = 023031\,\rotgap{3}$.
  Then the sequence $\toeplitz{\bpat} = 023031 3\, 023031 1\, \ldots$ is not additive: 
  $\nth{\toeplitz{\bpat}}{8} \ne \nth{\toeplitz{\bpat}}{4} + \nth{\toeplitz{\bpat}}{2}$.
\end{example}

We note that an additive sequence over $\nalph{n}$ (with $n \in \posnat$) is surjective
if and only if it contains the element $1_{\nalph{n}}$.

\begin{lemma}\label{lem:add:ghom}
  Let $n, m \in \posnat$,
  $\sigma \in \nalph{n}^\omega$ an additive sequence,
  and
  $c \in \nalph{m}$.
  If $c$ is divisible by $\gscale{n}{m}$,
  then $\gtimes{\sigma}{c} \in \nalph{m}^\omega$ is an additive sequence.
  If, moreover, the sequence $\sigma$ is surjective, then the converse direction holds as well.
\end{lemma}
\begin{proof}
  For all $i,j\in\nat$ we have 
  $\nth{(\gtimes{\sigma}{c})}{ij}
   = \gtimes{\nth{\sigma}{ij}}{c}
   = \gtimes{(\nth{\sigma}{i}+\nth{\sigma}{j})}{c}
   = \gtimes{\nth{\sigma}{i}}{c} + \gtimes{\nth{\sigma}{j})}{c}
   = \nth{(\gtimes{\sigma}{c})}{i} + \nth{(\gtimes{\sigma}{c})}{j}$
  since $\sigma$ is additive and by Lemma~\ref{lem:ghom}.
  
  For the converse direction, assume that the sequence $\sigma$ is surjective 
  and $\gtimes{\sigma}{c}$ additive.
  There exist $i,j \in \nat$ such that 
  $\nth{\sigma}{i} = 1$ and $\nth{\sigma}{j} = n-1$.
  Then $\nth{\sigma}{ij} = \nth{\sigma}{i} + \nth{\sigma}{j} = 0$.
  As a consequence $\gtimes{\nth{\sigma}{ij}}{c} = \gtimes{\nth{\sigma}{i}}{c} + \gtimes{\nth{\sigma}{j}}{c} = 0$.
  Then $\gtimes{(i +_\nat j)}{c} = \gtimes{n}{c} = 0$.
  Hence $n \cdot_\nat c \equiv 0 \pmod{m}$, and
  thus $c$ must be a multiple of $\gscale{n}{m} = \frac{m}{\gcd{n}{m}}$.
\end{proof}

\begin{lemma}\label{lem:cap:gap}
  Let $w \in \aalph^*$ such that $p = \length{w}+1$ is prime, and $d\in\aalph$.
  Then $w \,\rotgap{d} \in \CAP$ if and only if
  $\nth{w}{k} = \nth{w}{i} + \nth{w}{j}$
  for all $i,j,k$ with $0 < i,j,k < p$ such that $k \equiv i\cdot j \pmod{p}$.
\end{lemma}
\begin{proof}
  For the direction `$\Rightarrow$', let
  $\apat = w \,\rotgap{d} \in \CAP$.
  Let $0 < i,j,k < p$ such that $k \equiv i\cdot j \pmod{p}$.
  Then
  $\nth{w}{k} = \nth{\toeplitz{\apat}}{k} = \nth{\toeplitz{\apat}}{ij}$ 
  by Lemma~\ref{lem:toeplitz:recurrence} since $k \equiv i\cdot j \pmod{p}$.
  Moreover, we have 
  $\nth{\toeplitz{\apat}}{ij} 
   = \nth{\toeplitz{\apat}}{i} + \nth{\toeplitz{\apat}}{j} 
   = \nth{w}{i} + \nth{w}{j}$. 
  Hence $\nth{w}{k} = \nth{w}{i} + \nth{w}{j}$. 
   
  For the direction `$\Leftarrow$', let $\apat = w \,\rotgap{d}$,
  and assume
  $\nth{w}{k} = \nth{w}{i} + \nth{w}{j}$
  for all $0 < i,j,k < p$ such that $k \equiv i\cdot j \pmod{p}$.
  We show additivity of $\astr = \toeplitz{\apat}$.
  Let $n,m \in \posnat$, we distinguish two cases:
  \begin{enumerate}
    \item Case: $p \nmid nm$.
      There exist $n',m',k'\in\nalph{p-1}$ such that 
      $n' \equiv n \pmod{p}$, $m' \equiv m \pmod{p}$ and $k' \equiv nm \pmod{p}$.
      We have:
      \begin{align*}
        \nth{\astr}{nm} 
        & = \nth{\astr}{k'} & \text{by Lemma~\ref{lem:toeplitz:recurrence} and $k' \equiv nm \pmod{p}$} \\
        & = \nth{w}{k'} \\
        & = \nth{w}{n'} + \nth{w}{m'} \\
        & = \nth{\astr}{n'} + \nth{\astr}{m'} \\
        & = \nth{\astr}{n} + \nth{\astr}{m} & \text{by Lemma~\ref{lem:toeplitz:recurrence}}
      \end{align*}
    \item Case: $p \mid n$ or $p \mid m$.
      We use induction on the exponent of the prime factor $p$ in $nm$.
      For symmetry assume $p \mid n$ (the other case follows analogously). We have:
      \begin{align*}
        \nth{\astr}{nm} 
        & = \nth{\astr}{\frac{n}{p}m} + d & \text{by Lemma~\ref{lem:toeplitz:recurrence}} \\
        & = \nth{\astr}{\frac{n}{p}} + \nth{\astr}{m} + d& \text{by induction hypothesis} \\
        & = \nth{\astr}{n} + \nth{\astr}{m} & \text{by Lemma~\ref{lem:toeplitz:recurrence}}
      \end{align*}
  \end{enumerate}
  This concludes the proof.
\end{proof}

\begin{lemma}\label{lem:dlog:cap}
  The Toeplitz word $\toeplitz{\dlog{p,g} \,\rotgap{d}}$ is additive
  for every prime $p$, primitive root $g$ modulo $p$ and $d\in\nalph{p-1}$.
\end{lemma}
\begin{proof}
  Let $p$ be a prime, $g$ a primitive root modulo $p$, $d\in\aalph$,
  and $\apat = \dlog{p,g} \,\rotgap{d}$.
  For $\apat \in \CAP$ it suffices to check
  $\nth{\dlog{p,g}}{k} = \nth{\dlog{p,g}}{i} + \nth{\dlog{p,g}}{j}$
  for all $0 < i,j,k < p$ such that $k \equiv i\cdot j \pmod{p}$ 
  by Lemma~\ref{lem:cap:gap}.
  This property follows from $\dlog{p,g}$ being the discrete logarithm modulo $p$.
\end{proof}

\begin{theorem}\label{thm:dlog:additive:complete} %
  Let $p$ be prime, $g$ a primitive root modulo $p$, and $c, d\in\aalph$.
  Then we have that the pattern
  $(\gtimes{\dlog{p,g}}{c}) \,\rotgap{d} \in \CAP$
  if and only if
  $\gscale{p-1}{\card{\aalph}}$ divides $c$.
\end{theorem}
\begin{proof}
  First, we observe that $(*)$
  $\gtimes{\toeplitz{\dlog{p,g} \,\rotgap{d}}}{c} = \toeplitz{\gtimes{\dlog{p,g}}{c} \,\rotgap{\gtimes{d}{c}}}$,
  and by Lemma~\ref{lem:cap:gap} we have
  $(\gtimes{\dlog{p,g}}{c}) \,\rotgap{\gtimes{d}{c}} \in \CAP$
  if and only if
  $(\gtimes{\dlog{p,g}}{c}) \,\rotgap{d} \in \CAP$.
  For the direction `$\Leftarrow$', let 
  $c\in\aalph$ such that $\gscale{p-1}{\card{\aalph}}$ divides $c$.
  Then $\gtimes{\toeplitz{\dlog{p,g} \,\rotgap{d}}}{c}$
  is additive by Lemmas~\ref{lem:dlog:cap} and~\ref{lem:add:ghom},
  which implies the claim by $(*)$.
  For the direction `$\Rightarrow$', let $(\gtimes{\dlog{p,g}}{c}) \,\rotgap{d} \in \CAP$.
  By $(*)$ we obtain that $\gtimes{\toeplitz{\dlog{p,g} \,\rotgap{d}}}{c}$ is additive,
  and by Lemma~\ref{lem:cap:gap} together with surjectivity of $\toeplitz{\dlog{p,g} \,\rotgap{d}}$
  it follows that $\gscale{p-1}{\card{\aalph}}$ divides $c \in \aalph$.
\end{proof}

Theorem~\ref{thm:superthm} summarizes the results of this section.
\begin{theorem}\label{thm:superthm}
  For every $p\in\primes$, let $g(p)$ denote a primitive root modulo~$p$.  
  Then we have
  \begin{align*}
    \CAP 
      \stackrel{\text{\makebox(10,5){def}}}{=}
      & \displset{\apat\in\wrd{(\galph{\aalph})}}{\text{$\toeplitz{\apat}$ is additive and $\gapnr{\apat} = 1$}}
      \\
      \stackrel{\text{\makebox(10,0){}}}{=}
      &
        \big\{
           \,\Tpow{((\gtimes{\dlog{p,g(p)}}{c}) \,\rotgap{d})}{k} 
           \, \big\vert \,
            \text{$p\in\primes$, $k \ge 1$, $c,d\in\aalph$, $\gscale{p-1}{\card{\aalph}} \mid c$} \,
        \big\} \\
      & \cup \big\{\,0^{+} \, \gap \, 0^{*}\,\big\}
  \end{align*}
\end{theorem}
For additive $\pair{r}{1}$-type Toeplitz words 
we can now give the exact generator~$\ampy\funin\primes\to\aalph$.
\begin{corollary}\label{cor:toeplitz:additive:infpg}
  Let $\toeplitz{P}$ 
  be additive with $\gapnr{P}=1$, so that
  $P = (\gtimes{\dlog{p,g}}{c}) \,\rotgap{d}$ 
  as in Theorem~\ref{thm:dlog:additive:complete}.
  Then $\toeplitz{P} = \pgs{\ampy}$ holds,
  where $\ampy\funin\primes\to\aalph$ is defined for all $q\in\primes$ by
  \[
    \ampy(q) = 
    \begin{cases}
      \gtimes{\dlog{p,g}(i)}{c} & \text{if $q \equiv i \pmod{p}$, for some $i$ with $1 \le i < p$} \\
      d & \text{if $q = p$}
    \end{cases}
  \]
  The sequence $\toeplitz{P}$ is infinitely prime generated if and only if $c \ne 0$ and $p > 2$,
  by Dirichlet's Theorem on arithmetic progressions~\cite{diri:1837}.
\end{corollary}

\begin{example}
  The pattern $P = 001011\rotgap{1}$ derived from $\dlog{7,3}$
  yields an additive binary sequence 
  whose underlying (infinite) prime set is $X = \primes\cap\residues{3,5,6}{7}\cup\{7\}$,
  that is, $\toeplitz{P} = \pgs{X}$.
\end{example}

\begin{example}
  Toeplitz patterns for single prime generated sequences 
  have a simple shape: 
  Let $p\in\primes$, $d \in \aalph$ and define 
  $\funap{\ampy}{p} = d$ and $\funap{\ampy}{q} = 0$ for $q\ne p$.
  Then 
  $\pgs{\ampy} = \toeplitz{0^{p-1}\tsp\rotgap{d}}$ 
  by Corollary~\ref{cor:toeplitz:additive:infpg} (take $c=0_{\aalph}$).
  E.g., for the \pd{} sequence $\PD = \pgs{2}$ we have $\PD = \toeplitz{0\rotgap{1}}$.
\end{example}

\section{Toeplitz Permutations}\label{sec:toeplitz:permutations}
There is a strong connection between a Toeplitz pattern~$P$ 
and the arithmetic progressions through~$\toeplitz{P}$.
For $\gapnr{P} = 1$ we show that every subsequence $\toeplitz{P}_{a,b}$ 
is a Toeplitz word $T(P_{a,b})$ where $P_{a,b}$ is derived from $P$.
Thus the classification of the arithmetic self-similarity of $\pair{r}{1}$ Toeplitz words 
is reduced to a problem of analyzing patterns.
And, using the results of Section~\ref{sec:toeplitz:additive},
we conclude that $\AS{\astr} = \displset{\pair{b}{b}}{b\in\posnat}$ 
for all non-periodic additive Toeplitz words~$\astr\in\strp{\aalph}$ of type~$\pair{r}{1}$.

\begin{lemma}\label{lem:abc}
  Fix $b,r \in \mathbb{N}_{>0}$ s.t.\ $\gcd{b}{r} =1$ and $r \geq 2$. 
  Then $\exists\, c,m \in \mathbb{N}_{>0}$ s.t.\ $cb = r^{m} - 1$. 
\end{lemma}
\begin{proof}
      Suppose that $b \nmid r^{m}-1$ for all $ m \in \mathbb{N}_{>0}$. 
      Then by the pigeon hole principle there must be some $p,s,v \in \mathbb{N}_{>0}$ such that
        \begin{align*}
          r^{p}-1   & \equiv v  \pmod{b} \\
          r^{p+s}-1 & \equiv v  \pmod{b}  
        \end{align*}
      It follows that:
        \begin{align*}
            b &\mid r^{p+s} - r^{p}    \\
            b &\mid r^{p}(r^{s}-1)    \\
            b &\mid r^{s}-1
        \end{align*}
      So then $b$ divides $r^{s}-1$, contradicting the assumption.  
\end{proof}

In the sequel we index Toeplitz patterns $P$ (and sequences~$\toeplitz{P}$) by positive integers, 
i.e., $P = P(1) P(2) \ldots P(r)$ with $r = \length{P}$.
We write $P(k)$ to denote $P^\omega(k) = P(k \mod r)$ where we take 
$1,2,\ldots,r$ as the representatives of the congruence classes modulo $r$.

\begin{definition}\label{def:arith:toeplitz:permutation}
  Let $P$ be a one-gap $\aalph$-pattern of length $r \ge 2$ 
  with its single gap $f\in\symgroup{\aalph}$ at the end. 
  Fix $a,b \in \posnat$ such that $a \leq b$ and $\gcd{b}{r} =1$. 
  Let $c,m \in \posnat$ with $m$ minimal such that $cb = r^m - 1$ (Lemma~\ref{lem:abc}).
  The \emph{arithmetic permutation $P_{a,b}$ of $P$} is defined by
    \begin{align*} 
      P_{a,b} = \Tpow{P}{m}(a) \Tpow{P}{m}(a+b) \ldots \Tpow{P}{m}(a+b(r^{m}-1))
    \end{align*} 
\end{definition}
Note that the pattern $P_{a,b}$ is a \emph{permutation} of $\Tpow{P}{m}$ (so $\gapnr{P_{a,b}} = 1$)
because $b$ is a generator of the additive group~$\{1,2,\ldots,r^m\}$. 
The following lemma generalizes~\cite[Theorem~3.8]{dann:2011}.

Recall that we index Toeplitz patterns $P$ (and sequences~$\toeplitz{P}$) by positive integers, 
i.e., $P = P(1) P(2) \ldots P(r)$ with $r = \length{P}$.
We write $P(k)$ to denote $P^\omega(k) = P(k \mod r)$ where we take 
$1,2,\ldots,r$ as the representatives of the congruence classes modulo $r$.

\begin{lemma}\label{lem:arith:toeplitz:permutation}
  Let $P \in \newrd{(\gaalph)}$ be a pattern of length $r \ge 2$ with its single gap $\agap$ at the end.
  Let $a,b \in \posnat$ such that $a \leq b$ and $\gcd{b}{r} =1$.
  Then $\toeplitz{P}_{a,b} = \toeplitz{P_{a,b}}$.
\end{lemma}
\begin{proof}
  Let $c,m \in \posnat$ with $m$ minimal such that $bc = r^m - 1$ so that $P_{a,b}$ is the permutation
  \begin{align*} 
    P_{a,b} = \Tpow{P}{m}(a) \Tpow{P}{m}(a+b) \ldots \Tpow{P}{m}(a+(r^{m}-1)b)
  \end{align*} 
  of $\Tpow{P}{m} = \Tpow{P}{m}(1) \Tpow{P}{m}(2) \ldots \Tpow{P}{m}(r^m)$ 
  as in Definition~\ref{def:arith:toeplitz:permutation}.
  Recall that $T(P) = T(\Tpow{P}{m})$ (Lemma~\ref{lem:pattern:composition}),
  and that the single gap~$f^m$ of $\Tpow{P}{m}$ is at the end, 
  $\nth{\Tpow{P}{m}}{r^m} = f^m$.
  Let $j$ be the index $1 \leq j \leq r^{m}$ such that 
  $\nth{P_{a,b}}{j} = f^{m}$.
  From $\nth{P_{a,b}}{j} = \nth{\Tpow{P}{m}}{a+(j-1)b}$ 
  we infer $a+(j-1)b \equiv r^m \pmod{r^m}$.
  Note that $j$ is unique with this property among $1,2,\ldots,r^m$.
  We also have that $a + ((ac+1)-1))b = a + acb = a + a(r^m-1) = ar^m$
  with $1 \le ac + 1 \le r^m$, and 
  so $a+(j-1)b=ar^m$ and $j = ac + 1$.
  We prove \[ \nth{\toeplitz{P}_{a,b}}{n} = \nth{\toeplitz{P_{a,b}}}{n} \]
  by induction on $n\in\nat$. 
  (We tacitly make use of Lemma~\ref{lem:toeplitz:recurrence}.)

  Let $n=sr^{m}+i$ for some $s\in\nat$ and $1 \leq i \leq r^{m}$.
  Consider the base case $n \le r^m$, that is, $s = 0$.
  For $i \ne j$, the claim follows from the construction of $P_{a,b}$.
  For $i = j$ we find
  \begin{align*}
    \nth{\toeplitz{P}_{a,b}}{j} 
    & = \nth{\toeplitz{\Tpow{P}{m}}_{a,b}}{j} \\
    & = \nth{\toeplitz{\Tpow{P}{m}}}{a+(j-1)b} \\
    & = \nth{\toeplitz{\Tpow{P}{m}}}{ar^m} \\
    & = f^m(\nth{\toeplitz{\Tpow{P}{m}}}{a}) \\
    & = \nth{\toeplitz{P_{a,b}}}{j}
  \end{align*}
  Now let $s \ge 1$.
  For $i \ne j$ we obtain
    \begin{align*}
      \nth{\toeplitz{P}_{a,b}}{n}
      & = \nth{\toeplitz{P}}{a+b(n-1)} \\
      & = \nth{\toeplitz{P}}{a+b(n-1-r^m)} \\
      & = \nth{\toeplitz{P}_{a,b}}{n-r^m)} \\
      & \stackrel{\text{IH}}{=} \nth{\toeplitz{P_{a,b}}}{n-r^m} \\
      & = \nth{\toeplitz{P_{a,b}}}{n} 
    \end{align*}  
  If $i=j=ac+1$ we find 
    \begin{align*}
      \toeplitz{P}_{a,b}(n)
      & = \toeplitz{P}(a+b(n-1)) \\
      & = \toeplitz{P}(a+b(sr^m+i-1)) \\
      & = \toeplitz{P}(a+b(sr^m+ac)) \\
      & = \toeplitz{P}(a+acb+sbr^{m}) \\
      & = \toeplitz{P}(a+a(r^{m}-1)+sbr^{m})\\
      & = \toeplitz{P}((a+sb)r^{m}) \\
      & = f^m(\toeplitz{P}(a+sb)) \\
      & = f^m(\toeplitz{P}_{a,b}(s+1)) \\
      & \stackrel{\text{IH}}{=} f^m(\toeplitz{P_{a,b}}(s+1)) \\
      & = \nth{\toeplitz{P_{a,b}}}{sr^m + j} \\
      & = \nth{\toeplitz{P_{a,b}}}{n} \qedhere
    \end{align*}  
\end{proof}

\begin{example}
  Let $P = 0123f$ with $f(n)= n+2 \mod{5}$ and $\astr = \toeplitz{P}$. 
  Then $\astr_{2,2} = \toeplitz{P_{2,2}}$ with $P_{2,2} = 1302f$,
  and $\astr_{2,12} = \toeplitz{P_{2,12}}$ with $P_{2,12} = 1302 f^{2}  1302 0 1302 2 1302 4 1302 1$;
  for the latter note that $cb = 2 \cdot 12 = 5^2 - 1 = r^{m} - 1$ and  
  $P^{(2)} = 0123 2 0123 4 0123 4 0123 0 0123 f^2$.
\end{example}

\begin{lemma}\label{lem:arith:toeplitz:gap:middle}
  Let $Q$ be a one-gap Toeplitz pattern of length $r$ of the form $Q  = u f v$
  for some $f \in \symgroup{\aalph}$, $u \in \Sigma^{+}$ and $v \in \Sigma^{*}$.
  Then $Q = P_{a,b}$ for $P = \reverse{u} \reverse{v} f$, $a = \length{u} \le b = \length{Q}-1$.
\end{lemma}
\begin{proof}
  Immediate from the definitions.
\end{proof}

\begin{definition}\label{def:arith:toeplitz:gap:b:divides:r}
  Let $P = ufv$ with $u\in\newrd{\aalph}$, $f\in\symgroup{\aalph}$, and $v\in\wrd{\aalph}$.
  Let $j = \length{u} + 1$ the index of the single gap~$f$. 
  Let $a,b\in\posnat$ s.t.\ $a < j$ and $cb = r = \length{P}$ for some $c\in\posnat$,
  and define
  \begin{align*}
    B       & = P(a) P(a + b) \ldots P(a+(c-1)b) \\
    P_{a,b} & =
    \begin{cases}
      (\Tcomp{B}{u}) \; (\gap^{j'-1}  f \, \gap^{c-j'}) \; (\Tcomp{B}{v})
        & \text{if $\myex{j'}{j \equiv_r a + (j'-1)b}$} \\ 
      B & \text{otherwise (so $f$ is not in $B$)}
    \end{cases}
  \end{align*}
  In the first case ($\myex{j'}{j \equiv_r a + (j'-1)b}$) we have $B(j') = f$. 
  We also note that $b\length{B} = r$ and $b\length{P_{a,b}} = r^2$.
\end{definition}

\begin{aplemma}\label{lem:pat:fg}
  Let $f,g\in\symgroup{\aalph}$ such that $\funcomp{f}{g} = \funcomp{g}{f}$,
  and $P = ufv$ a pattern of length $r$
  for some $u\in\newrd{\aalph}$ and $v\in\wrd{\aalph}$.
  Define $P^g = (\Tcomp{g}{u}) \, f \, (\Tcomp{g}{v})= \funap{g}{u} \, f \, \funap{g}{v}$.
  Then $\toeplitz{P^g} = \funap{g}{\toeplitz{P}}$.
\end{aplemma}
\begin{proof}
  We show $\nth{\toeplitz{P^g}}{n} = \nth{\funap{g}{\toeplitz{P}}}{n}$
  by induction on $n\in\posnat$.
  Let $j = \length{u}+1$ (so $P(j) = P^g(j) = f$).
  Let $n = rn'+i$ for some $n'\in\nat$ and $1 \le i \le r$.
  Using the recurrence relations for Toeplitz words (Lemma~\ref{lem:toeplitz:recurrence}) 
  we obtain:
  \begin{align*}
    \nth{\toeplitz{P^g}}{n} & = \nth{P^g}{i} = \funap{g}{\nth{P}{i}} 
      = \funap{g}{\nth{\toeplitz{P}}{i}} 
      && \text{if $i \ne j$} \\
    \nth{\toeplitz{P^g}}{n} & = \funap{f}{\nth{\toeplitz{P^g}}{n'+1}} 
      \stackrel{\text{IH}}{=} \funap{f}{\funap{g}{\nth{\toeplitz{P}}{n'+1}}} \\
      & = \funap{g}{\funap{f}{\nth{\toeplitz{P}}{n'+1}}} 
      = \funap{g}{\nth{\toeplitz{P}}{n}} 
      && \text{if $i = j$}
      \qedhere
  \end{align*}
\end{proof}

\begin{lemma}\label{lem:arith:toeplitz:gap:b:divides:r}
  Let $Q$ be a one-gap pattern, and $a,b\in\posnat$ s.t.\ $b$ divides $\length{Q}$.
  Then $\toeplitz{Q}_{a,b} = \toeplitz{(\Tpow{Q}{m})_{a,b}}$ where $m$ is minimal such that
  the (single) gap in $\Tpow{Q}{m}$ is at index $j$ with $j \gt a$.
\end{lemma}
\begin{proof}
  Let $P = \Tpow{Q}{m} = ufv$ of length $r = \length{Q}^m$ 
  for some $u\in\newrd{\aalph}$, $f\in\symgroup{\aalph}$, and $v\in\wrd{\aalph}$,
  and let $j  \length{u}+1$ be the index of the gap in $P$.
  Note that $\toeplitz{P} = \toeplitz{Q}$ by Lemma~\ref{lem:pattern:composition} 
  and $P_{a,b} = Q_{a,b}$ by definition.
  Let $n\in\posnat$.
  We prove $\nth{\toeplitz{P}_{a,b}}{n} = \nth{\toeplitz{P_{a,b}}}{n}$.
  
  Let $c\in\posnat$ be such that $bc = r$ and 
  define $B = P(a) P(a + b) \ldots P(a+(c-1)b)$. 
  We distinguish the following cases: 
  \begin{enumerate}
    \item 
      If $j \nequiv a + (j'-1)b \pmod{r}$ for all $j'\in\posnat$,
      then $P_{a,b} = B$, and so $\toeplitz{P_{a,b}} = B^\omega$ because $B$ is free of gaps.
      By Lemma~\ref{lem:toeplitz:recurrence} we obtain
      \begin{align*}
        \nth{\toeplitz{P}_{a,b}}{n} 
          & = \nth{\toeplitz{P}}{a + (n-1)b} \\
          & = P(a + (n-1)b) \\
          & = P_{a,b}(n)  = B(n) \\
          & = \nth{\toeplitz{P_{a,b}}}{n}
      \end{align*}
    \item 

      Assume there is $j'$ with $1 \le j' \le c$ such that $j \equiv a + (j'-1)b \pmod{r}$,
      so that $B(j') = f$.
      Let $r' = \length{P_{a,b}} = rc = \frac{r^2}{b}$.
      We make a further case distinction.
      \begin{enumerate}
        \item 
          Assume $n \nequiv j' \pmod c$.
          Let $n = r'n' + i$ for some $n'\in\nat$ and $1 \le i \le r'$, so that $i \nequiv j' \pmod{c}$.
          We prove $\nth{\toeplitz{P_{a,b}}}{r' x + i'} = B(i')$ for all $i' \equiv i \pmod{c}$ 
          by induction on $x\in\nat$.
          Using Lemma~\ref{lem:toeplitz:recurrence} we obtain
          (note that $\gapnr{\Tcomp{B}{x}}=0$ for all $x\in\wrd{\aalph}$)
          \begin{align*}
            \nth{\toeplitz{P_{a,b}}}{r' x + i'} 
              & = \nth{(\Tcomp{B}{u})}{i'} 
                = B(i')
              && \text{if $1 \le i' \le c \length{u}$} \\
            \nth{\toeplitz{P_{a,b}}}{r' x + i'} 
              & = \nth{(\Tcomp{B}{v})}{i''} 
                = B(i')
              && \text{if $i' = c \length{u} + c + i''$} \\
              &&& \text{for some $1 \le i'' \le c \length{v}$}
          \end{align*}
          For $i' = c \length{u} + i''$ for some $1 \le i'' \le c$ we find with Lemma~\ref{lem:toeplitz:recurrence}
          \begin{align*}
            \nth{\toeplitz{P_{a,b}}}{r' x + i'} 
              & = \nth{\toeplitz{P_{a,b}}}{c x + i''} \\
              & = \nth{\toeplitz{P_{a,b}}}{r' x' + i'''} && \text{for some $x'\in\nat$, and $i''' \equiv_c i$} \tag{$\ast$}\label{eq:'''} \\
              & \stackrel{\text{IH}}{=} B(i''') \\
              & = B(i')
          \end{align*}
          The identity~\eqref{eq:'''} holds because $i'' \equiv i' \equiv i \pmod{c}$ and $r' \equiv 0 \pmod{c}$.

          For $\nth{\toeplitz{P}_{a,b}}{n}$ we find,
          again with Lemma~\ref{lem:toeplitz:recurrence}
          \begin{align*}
            \nth{\toeplitz{P}_{a,b}}{r' n' + i}
              & = \nth{\toeplitz{P}}{a + (r' n' + i - 1)b} \\
              & = \nth{\toeplitz{P}}{r^2 n' + a + (i - 1)b} \\
              & = \nth{P}{a + (i - 1)b} 
              & = B(i)
          \end{align*}
          For the third step, note that we have $a + (i-1)b \nequiv_r j$  since $i \nequiv_c j'$.
          Hence we have shown $\nth{\toeplitz{P}_{a,b}}{n} = \nth{\toeplitz{P_{a,b}}}{n}$.
          
        \item 
          For $n \equiv j' \pmod c$, let $n = tc + j'$ for some $t \in \nat$. 
          We have that
          \begin{align*}
            \nth{\toeplitz{P}_{a,b}}{n}
            & = \nth{\toeplitz{P}}{a + (n-1)b} \\
            & = \nth{\toeplitz{P}}{a + (tc + j'-1)b} \\
            & = \nth{\toeplitz{P}}{rt + a + (j'-1)b} \\
            & = \nth{\toeplitz{P}}{rt + j} \\
            & = \funap{f}{\nth{\toeplitz{P}}{t + 1}}
          \end{align*}
          We also have
          \begin{align*}
            \nth{P_{a,b}}{n}
            & = G(t+1) \dquad \text{where $G = (\Tcomp{f}{u}) \,f\, (\Tcomp{f}{v}) = \funap{f}{u}\,f\,\funap{f}{v}$}
          \end{align*}
          because $\nth{(\Tcomp{B}{x})}{kc+j'} = (\Tcomp{f}{x})(k+1)$ 
          for all $x\in\wrd{\aalph}$, $k\in\nat$,
          and $\nth{(\gap^{j'-1}  f \, \gap^{c-j'})}{j'} = f$.
          
          We show $\nth{\toeplitz{P_{a,b}}}{xc + j'} = \nth{\toeplitz{G}}{x+1}$ by induction on $x$.
          Next we distinguish two cases for $x$.
          If $x \nequiv \length{u} \pmod{r}$, then
          \begin{align*}
            \nth{\toeplitz{P_{a,b}}}{xc + j'} 
            & = \nth{P_{a,b}}{xc + j'} && \text{since $P_{a,b}(xc + j') \in \aalph$} \\
            & = G(t+1) && \text{since $G(x+1) \in \aalph$} \\
            & = \nth{\toeplitz{G}}{x+1}
          \end{align*}
          If $x \equiv \length{u} \pmod{r}$, then $x = x' r + \length{u}$ for some $x'\in\nat$, and
          \begin{align*}
            \nth{\toeplitz{P_{a,b}}}{xc + j'} 
            & = \nth{\toeplitz{P_{a,b}}}{(x' r + \length{u})c + j'} \\
            & = \nth{\toeplitz{P_{a,b}}}{x' r' + \length{u}c + j'} \\
            & = \funap{f}{\nth{\toeplitz{P_{a,b}}}{cx' + j'}} \\
            & \stackrel{\text{IH}}{=} \funap{f}{\nth{\toeplitz{G}}{x' + 1}} \\
            & = \nth{\toeplitz{G}}{x + 1}
          \end{align*}
          We have shown $\nth{\toeplitz{P}_{a,b}}{n} = \funap{f}{\nth{\toeplitz{P}}{t + 1}}$,
          and $\nth{\toeplitz{P_{a,b}}}{n} = \nth{\toeplitz{G}}{t+1}$.          
          It thus remains to be shown that 
          $\funap{f}{\nth{\toeplitz{P}}{t + 1}} = \nth{\toeplitz{G}}{t+1}$, 
          which follows from Lemma~\ref{lem:pat:fg}.
        \end{enumerate}
  \end{enumerate}
\end{proof}

\begin{example}
  Let $Q = 0g12$ with $g(n) = n + 1 \mod 3$.
  We construct $Q_{2,2}$ as follows:
  First take $\Tpow{Q}{2} = 0112 0g^{2}12 0212 0012$
  so that the gap index 
  $j = 6 > a = 2$.
  Then 
  we find 
  \[
    Q_{2,2} = 
    (\Tcomp{B}{0112 0}) (\gap\gap g^2\gap \gap\gap \gap\gap) (\Tcomp{B}{12 0212 0012})
    \dquad\text{where $B = 1 2 g^{2} 2 2 2 0 2$}
  \]
\end{example}

\begin{theorem}\label{thm:arith:toeplitz}
  Every arithmetic subsequence of an $\pair{r}{1}$-type Toeplitz word is a Toeplitz word.
\end{theorem}
\begin{proof}
  Let $P$ be a pattern with $r = \length{P}$ and $\gapnr{P}=1$,
  and let $a,b\in\posnat$.
  Without loss of generality:
  \begin{enumerate}
    \item 
      We assume the gap of $P$ to be at the end.
      For otherwise, using Lemma~\ref{lem:arith:toeplitz:gap:middle} we have
      a pattern~$Q$ of length $r$ with its gap at the end, 
      and $1 \le c \le r-1$ such that $Q_{c,r-1} = P$.
      By Lemma~\ref{lem:arith:toeplitz:permutation} we have that $\toeplitz{Q}_{c,r-1} = \toeplitz{P}$,
      and hence $\toeplitz{P}_{a,b} = \toeplitz{Q}_{c+(r-1)(a-1),b(r-1)}$.
      Then we proceed with constructing a pattern generating $\toeplitz{Q}_{c+(r-1)(a-1),b(r-1)}$.
    \item 
      We assume $b < r$, since $\toeplitz{P} = \toeplitz{\Tpow{P}{k}}$ 
      for all $k \ge 1$ by Lemma~\ref{lem:pattern:composition}.
  \end{enumerate}

  Let $b = b_1 b_2$ such that $b_1$ is maximal with $\gcd{b_1}{r} = 1$.
  Note that ($\ast$) $b_1$ contains all primes in the factorization of $b$ that
  do not occur in the prime factorization of $r$.
  Moreover let $a_2 \in \posnat$ such that $1 \le a_1 \le b_1$ where $a_1 = a - (a_2-1) b_1$.
  Then
  \begin{align*}
    \toeplitz{P}_{a,b} = (\toeplitz{P}_{a_1,b_1})_{a_2,b_2}
  \end{align*} 
  by composition of arithmetic progressions.
  The sequence $\toeplitz{P}_{a_1,b_1}$ is generated by the Toeplitz pattern $P_{a_1,b_1}$ 
  of length $r^m$ with one gap by Lemma~\ref{lem:arith:toeplitz:permutation} (for some $m \ge 1$). 
  Thus it suffices to show that $\toeplitz{P_{a_1,b_1}}_{a_2,b_2}$ is generated by a Toeplitz pattern.
  By ($\ast$) all primes in the factorization of $b_2$ occur in $r$, 
  so also in $r^m$. Hence there exists $n \ge 1$ such that $b_2$ divides $(r^m)^n = r^{mn}$. 
  Since $\toeplitz{P}_{a_1,b_1} = \toeplitz{\Tpow{P_{a_1,b_1}}{n}}$, 
  $\gapnr{\Tpow{P_{a_1,b_1}}{n}} = 1$, and $\length{\Tpow{P_{a_1,b_1}}{n}} = r^{nm}$ 
  the claim follows by Lemma~\ref{lem:arith:toeplitz:gap:b:divides:r}.
\end{proof}

%

\section{Keane Words}\label{sec:keane}
We generalize (binary) \emph{generalized Morse sequences} as introduced by Keane in~\cite{kean:1968}, 
to sequences over the cyclic additive group $\aalph$, but restrict to `uniform' infinite block products
$u \times u \times u \times \cdots$, or `Keane words' as we call them.
\begin{definition}\label{def:keane:product}
  The \emph{Keane product} is the binary operation~$\sKmult$ on $\wrd{\aalph}$
  defined as follows:
  \begin{align*}
    \Kmult{u}{\wrdemp} & = \wrdemp
    &&&
    \Kmult{u}{av} & = (u + a)(\Kmult{u}{v})
  \end{align*}
  for all $u,v\in\wrd{\aalph}$ and $a\in \aalph$.
  We define $\Kpow{u}{n}$ by $\Kpow{u}{0} = 0$ and 
  $\Kpow{u}{n+1} = \Kmult{u}{\Kpow{u}{n}}$.
  
  Let $u\in\wrd{\aalph}$ with $\length{u} \ge 2$ and $\nth{u}{0} = 0$.
  The \emph{Keane word generated by $u$} is defined by
  \[ \keane{u} = \lim_{n\to\infinity} \Kpow{u}{n} \]
\end{definition}
The product $\Kmult{u}{v}$ is formed by concatenation
of $\length{v}$ copies of $u$
(so $\length{\Kmult{u}{v}} = \length{u}\cdot\length{v}$),
taking the $i$th copy as $u + \nth{v}{i}$
($0\leq i\leq \length{v}-1$).
As $0$ is the identity with respect to the $\sKmult$-operation,
$\Kmult{u}{v}$ is a proper extension of $u$
whenever $\nth{v}{0} = 0$ and $\length{v} \ge 2$.
Hence $\keane{u}$ is well-defined and is the unique infinite fixed point of
$x \mapsto u \times x$.
Also note that $\keane{u}$ is the iterative limit of the $\length{u}$-uniform morphism~$h$ 
defined by $h(a) = u + a$, for all $a\in\aalph$;
thus Keane words are automatic sequences~\cite{allo:shal:2003}.

\begin{proposition}\label{prop:keane:monoid}
    $\triple{\wrd{\aalph}}{\sKmult}{0}$ is a monoid.
\end{proposition}
Note that $\sKmult$ is not commutative, e.g.,
$\Kmult{00}{01} = 0011$ whereas $\Kmult{01}{00} = 0101$. 

\begin{lemma}\label{lem:keane:nth}
  Let $u \in \wrd{\aalph}$ 
  with $u(0) = 0$ and $k = \length{u} \ge 2$, and
  let $0 \le i \lt k$.
  Then for all $n \ge 0$ we have $\nth{\keane{u}}{nk+i} = \nth{\keane{u}}{n} + \nth{u}{i}$,
  and so $\asub{\keane{u}}{i}{k} \similarx \keane{u}$.
  Hence (by Lemma~\ref{lem:AS:closed:under:compostion})
  $\pair{i}{k^n} \in \AS{\keane{u}}$ 
  for all $n\ge 0$ and $0 \le i \lt k^n$.
\end{lemma}
\proof
Because of associativity of $\sKmult$
we have that 
$\Kmult{\ablk}{\Kpow{\ablk}{n}} = \Kmult{\Kpow{\ablk}{n}}{\ablk}$,
which means that we can view $\Kpow{\ablk}{n+1}$ as consisting of
$\length{\ablk}^n$ variants of $\ablk$, as well as 
$\length{\ablk}$ variants of $\Kpow{\ablk}{n}$, respectively:
\begin{align}
  & \Kpow{\ablk}{n+1}
    = \Kmult{\ablk}{\Kpow{\ablk}{n}}
    = (\ablk + \nth{\Kpow{\ablk}{n}}{0}) \cdots (\ablk + \nth{\Kpow{\ablk}{n}}{\length{u}^n-1})
    \label{eq:keane:left}
  \\
  & \Kpow{\ablk}{n+1}
    = \Kmult{\Kpow{\ablk}{n}}{\ablk}
    = (\Kpow{\ablk}{n} + \nth{\ablk}{0}) \cdots (\Kpow{\ablk}{n} + \nth{\ablk}{\length{\ablk}-1})
    \label{eq:keane:right}
\end{align}
  Let $0 \le i \lt k$ and $0 \le j \lt k^n$.
  Since we have $\Kpow{\ablk}{n} \sqsubseteq \keane{\ablk}$ for all $n \geq 0$,
  by the use of \eqref{eq:keane:left} and~\eqref{eq:keane:right}
  we conclude
  \begin{align*}
    \nth{\Kpow{\ablk}{n+1}}{(j\length{\ablk}+i} 
    & = \nth{(\ablk + \nth{\Kpow{\ablk}{n}}{j})}{i}
      = \nth{\Kpow{\ablk}{n}}{j} + \nth{\ablk}{i}
    \\
    \nth{\Kpow{\ablk}{n+1}}{i\length{\ablk}^n+j} 
    & = \nth{(\Kpow{\ablk}{n} + \nth{\ablk}{i})}{j} 
      = \nth{\Kpow{\ablk}{n}}{j} + \nth{\ablk}{i} 
  \qedhere
  \end{align*}

\begin{proposition}\label{prop:keane:additive}
  The only completely additive Keane word is the constant zero sequence.
\end{proposition}
\begin{proof}
  Let $u= u_0 u_1 \ldots u_{k-1} \in \wrd{\aalph}$ 
  with $u_0 = 0$ and $k \ge 2$,
  and assume that $\astr = \keane{u} \in \str{\aalph}$ is additive
  (note that, as $\astr$ is indexed $0,1,\ldots$, additivity of $\astr$ means 
  $\nth{\astr}{nm-1} = \nth{\astr}{n-1} + \nth{\astr}{m-1}$ for all $n,m > 0$).
  Then $\asub{\astr}{k}{k+1} = \astr + \nth{\astr}{k}$, that is,
  \begin{align}
    \nth{\asub{\astr}{k}{k+1}}{n}
    = \nth{\astr}{k+(k+1)n} 
    = \nth{\astr}{k(n+1)+n} 
    = \nth{\astr}{n} + \nth{\astr}{k}
    \dquad (n\in\nat) 
    \label{eq:prop:keane:additive:1}
  \end{align}
  Using the recurrence relations for Keane words (Lemma~\ref{lem:keane:nth})
  we obtain
  \begin{align}
    \nth{\astr}{k(i+1)+i} = \nth{\astr}{i+1} + u_i
    \dquad (0 \le i \lt k)
    \label{eq:prop:keane:additive:2}
  \end{align}
  Combining \eqref{eq:prop:keane:additive:1} and \eqref{eq:prop:keane:additive:2}
  gives $\nth{\astr}{i+1} = \nth{\astr}{k}$
  for all $0 \le i \lt k$.
  Hence $u_1 = u_2 = \ldots = u_{k-1}$.

  To see that $u_i = 0$ for all $0 \le i < k$,
  we show $u_1 = u_1 + u_1$ in each of the following three cases:
  For $k \ge 4$ this is immediate: $u_3=u_1+u_1$ follows from additivity of $\astr$.
  For $k=3$ we get $\nth{\astr}{3} = u_0 + u_1 = u_1$ from equation~\eqref{eq:prop:keane:additive:2}
  and $\nth{\astr}{3} = u_1 + u_1$ by additivity.
  Finally, for $k=2$, we get $\nth{\astr}{8} = u_1$ and $\nth{\astr}{2} = u_1$,
  since by Lemma~\ref{lem:keane:nth} we have 
  $\nth{\keane{v}}{\ell^n} = v_1$ for all blocks~$v = v_0 v_1 \ldots v_{\ell}$ and $n\in\nat$.
  Moreover, by additivity, $\nth{\astr}{8} = \nth{\astr}{2} + \nth{\astr}{2} = u_1 + u_1$.
\end{proof}

In what follows we let $(n)_k$ denote the base \hyph{$k$}{expansion} of $n$,
and $\nrofin{a}{w}$ the number of occurrences of a letter $a\in \aalph$ in a word $w\in \wrd{\aalph}$.
\begin{theorem}\label{thm:keane:digits}
  Let $u = u_0 u_1 \ldots u_{k-1} \in \wrd{\aalph}$ 
  be a $k$-block with $u_0 = 0$ and $k \geq 2$.
  Then
  \[
    \keane{u}(n) 
    = \summ{i \in \nalph{k}}{}{\gtimes{\nrofin{i}{(n)_k}}{u_i}}
    \dquad (n \geq 0)
  \]
\end{theorem}
\proof
We prove 
\(
  \myall{n\in\nat}{
    \length{(n)_k} = r 
    \implies 
    \nth{\keane{\ablk}}{n} = \summ{i \in \nalph{k}}{}{\gtimes{\nrofin{i}{(n)_k}}{\ablk_i}}
  }
\)
by induction on $r \in \posnat$.
If $r=1$, then $n\in\nalph{k}$ and $\nth{\keane{\ablk}}{n} = u_n$.
For $r > 1$, 
let $(n)_k = n_{r-1} n_{r-2} \ldots n_{1} n_{0}$
and define $n' = \frac{n - n_0}{k}$.
Then $(n')_k = n_{r-1} n_{r-2} \ldots n_{1}$ and $\length{(n')_k} = r-1$,
and we find
\begin{align*}
  \nth{\keane{\ablk}}{n}
  & = \nth{\keane{\ablk}}{n'k+n_0} \\
  & = \nth{\keane{\ablk}}{n'} + \ablk_{n_0} && \text{(Lemma~\ref{lem:keane:nth})} \\
  & \stackrel{\text{IH}}{=} \left(\summ{i \in \nalph{k}}{}{\gtimes{\nrofin{i}{(n')_k}}{\ablk_i}}\right) + \ablk_{n_0} \\
  & = \summ{i \in \nalph{k}}{}{\gtimes{\nrofin{i}{(n)_k}}{\ablk_i}}
  && \qedhere
\end{align*}

\begin{example}
  For the Morse sequence $\morse = \keane{01} = 01 \times 01 \times \cdots$, 
  Theorem~\ref{thm:keane:digits} gives another well-known definition of~$\morse$,
  due to J.H.~Conway~\cite{conw:1976}:
  $\nth{\morse}{n}$ is the parity of the number of $1$s in the binary expansion of $n$.
  Likewise, for the generalized Morse sequence $\mephisto = \keane{001}$ discussed in~\cite{kean:1968,jaco:1992},
  and called the `Mephisto Waltz' in~\cite[p.~105]{jaco:1992}, 
  we find that $\nth{\mephisto}{n}$
  is the parity of the number of $2$s in the ternary expansion of $n$.
\end{example}

Let $x = x_0 x_1 \ldots x_{k-1}$ 
with $k \ge 1$ be a word over $\aalph$.
The \emph{first difference}~$\diff{x}$ of $x$
is defined by $\diff{x} = \wrdemp$ if $\length{x} = 1$ 
and $\diff{x} = (x_1 - x_0)(x_2 - x_1) \ldots (x_{k-1} - x_{k-2})$,
otherwise.

We give an embedding of the Keane monoid (Prop.~\ref{prop:keane:monoid})
into the monoid of Toeplitz pattern composition (Prop.~\ref{prop:Tpat:monoid}).

Let $\mcl{B} = \displset{u\in\newrd{\aalph}}{u_0 = 0}$,
and define the map $\sTdiff \funin \mcl{B} \to \wrd{(\galph{\aalph})}$ by
\begin{align*}
  \Tdiff{u} & = \diff{u} \rotgap{d} 
  & 
  \text{for all $u \in \mcl{B}$ and $d = u_0 - u_{\length{u}-1} = - u_{\length{u}-1}$}
\end{align*}
\begin{theorem}\label{thm:Tdiff:embedding}
  \( 
    \sTdiff 
    \funin
    \triple{\newrd{\aalph}}{\sKmult}{0} 
    \hookrightarrow 
    \triple{\newrd{(\galph{\aalph})}}{\sTcmp}{\gap}
  \)
  is an injective homomorphism:
  \begin{align*}
     \Tdiff{0} & = \gap &
     \Tdiff{\Kmult{u}{v}} & = \Tcomp{\Tdiff{u}}{\Tdiff{v}}
  \end{align*}
\end{theorem}
\begin{proof}
  \newcommand{\sTdiffa}[1]{\Delta_{\stoeplitz,#1}} 
  \newcommand{\Tdiffa}[1]{\funap{\sTdiffa{#1}}}
  We first show that $\sTdiff$ preserves structure. 
  The preservation of identity is immediate.
  For $c\in\aalph$ let $\sTdiffa{c} \funin \newrd{\aalph} \to \newrd{(\galph{\aalph})}$
  be defined by $\Tdiffa{c}{x} = \diff{x} \, \rotgap{c}$ for all $x\in\newrd{\aalph}$
  (so that $\Tdiff{x} = \Tdiffa{c}{x}$ iff $c = - x_{\length{x}-1}$, for all $x\in\mcl{B}$).
  We prove the following equation, for all $u,v\in\newrd{\aalph}$ and $d\in\aalph$:
  \[
    \Tdiffa{d'}{\Kmult{\ablk}{\bblk}} 
    = \Tcomp{\Tdiff{\ablk}}{\Tdiffa{d}{\bblk}}
     \dquad 
     \text{where $d' = d_u + d$ and $d_u = \ablk_0 - \ablk_{\length{\ablk} - 1}$,}
  \]
  by induction on the length of $\bblk$.
  The claim 
  then follows by taking $d = d_v = v_0 - v_{\length{v}-1} = - v_{\length{v}-1}$
  since $\Tdiff{\Kmult{\ablk}{\bblk}} = \Tdiffa{d_u+d_v}{\Kmult{\ablk}{\bblk}}$.

  If $\length{\bblk}  = 1$, then $\bblk = a$ for some $a \in \aalph$, 
  and 
  \begin{align*}
    \Tdiffa{d'}{\Kmult{\ablk}{a}} 
    & = \Tdiffa{d'}{\ablk+a} \\
    & = \diff{\ablk+a} \rotgap{d'} \\
    & = \diff{\ablk} \rotgap{d'} \\
    & = \Tcomp{(\diff{\ablk}\rotgap{d_u})}{\rotgap{d}} \\
    & = \Tcomp{\Tdiff{\ablk}}{\Tdiffa{d}{a}} 
  \end{align*}
  
  If $\length{\bblk}  > 1$, then $\bblk = a a' \bblk'$ for some $a,a' \in \aalph$
  and $\bblk'\in\wrd{\aalph}$, and we find
  \begin{align}
    \Tdiffa{d'}{\Kmult{\ablk}{aa'\bblk'}}
    & = \Tdiffa{d'}{(\ablk+a)(\Kmult{\ablk}{a'\bblk'})} \notag \\
    & = \diff{(\ablk+a)(\Kmult{\ablk}{a'\bblk'})}\rotgap{d'} \notag \\
    & = \diff{\ablk+a} (\ablk_0 + a' - (\ablk_{\length{u}-1} + a)) \diff{\Kmult{\ablk}{a'\bblk'}} \rotgap{d'} \label{eq:aa} \\
    & = \diff{\ablk} (d_u + a' - a) \Tdiffa{d'}{\Kmult{\ablk}{a'\bblk'}} \notag \\
    & = \diff{\ablk} (d_u + a' - a) (\Tcomp{\Tdiff{\ablk}}{\Tdiffa{d}{a'\bblk'}}) \tag{\text{IH}} \\
    & = \Tcomp{(\diff{\ablk}\rotgap{d_u})}{((a' - a)\Tdiffa{d}{a'\bblk'})} \notag \\
    & = \Tcomp{\Tdiff{\ablk}}{\Tdiffa{d}{aa'\bblk'}} \notag
  \end{align}
  where \eqref{eq:aa} follows from
  $\diff{xy} = \diff{x} (y_0 - x_{\length{x}-1}) \diff{y}$ for all $x,y\in\newrd{\aalph}$,
  and $\nth{(\Kmult{\ablk}{a'\bblk'})}{0} = \ablk_0 + a'$.

  Finally, we show that $\sTdiff$ is injective by defining a retraction 
  $\sTsum \funin \Tdiff{\mcl{B}} \to \mcl{B}$.
  Let $\apat$ be a $\aalph$-pattern $P\in\Tdiff{\mcl{B}}$.
  Then $\apat$ has the form 
  $P = a_1 a_2 \ldots a_n$ where
  $a_i \in \aalph$ for $1 \le i \lt n$ and 
  $a_n = \rotgap{d}$ with $d = - \sum_{1 \le i \lt n} a_i$.
  Define $\Tsum{P} = \ablk_0 \ablk_1 \ldots \ablk_{n-1}$ by 
  $\ablk_0 = 0$, $\ablk_i = \ablk_{i-1} + a_i$ ($1 \le i \le n-1$).

  We show that $\sTsum$ is a left-inverse of $\sTdiff$.
  Let $u = \ablk_0 \ablk_1 \ldots \ablk_{n-1} \in \mcl{B}$.
  Then $\Tdiff{\ablk} = \diff{\ablk} \rotgap{d}$ 
  with $d = - \ablk_{n-1}$.
  Let $\diff{\ablk} = a_1 a_2 \ldots a_{n-1}$.
  Then $a_i = \ablk_i - \ablk_{i-1}$ ($1 \le i \lt n$)
  and so $\Tsum{\Tdiff{u}} = u$.
\end{proof}

It follows that the first difference sequence of a Keane word is a Toeplitz word.
Here the difference operator $\sdiff$ is extended to 
$\llist{\aalph}\to\llist{\aalph}$ in the obvious way.
\begin{theorem}\label{thm:diff:keane}
  $\diff{\keane{u}} = \toeplitz{\Tdiff{u}}$, 
  for all blocks $u\in\mcl{B}$. 
\end{theorem}
\begin{proof}
  Let $u = u_0 u_1 \ldots u_{k-1}$ be a $k$-block over $\aalph$ with $u_0 = 0$ and $k \ge 2$.
  Let $n \in \posnat$ be a positive integer
  and $(n)_k = n_{r-1} n_{r-2} \ldots n_0$ the base $k$-expansion of $n$.
  
  The first nonzero digit of $(n)_k$ (reading from right to left) 
  is $n_a$ with $a = \padval{k}{n}$
  and so
  \[
    \begin{array}{rccccccccc}
      (n)_k   & = & n_{r-1} & n_{r-2} & \ldots & n_a     & 0     & 0     & \ldots & 0 \\
      (n-1)_k & = & n_{r-1} & n_{r-2} & \ldots & (n_a-1) & (k-1) & (k-1) & \ldots & (k-1)
    \end{array}
  \]
  Let $\kappa = \keane{u}$.
  From Theorem~\ref{thm:keane:digits} we then obtain 
  \begin{align*}
    \nth{\diff{\kappa}}{n} 
    & = \nth{\kappa}{n} - \nth{\kappa}{n-1} \\
    & = ( u_{n_a} + a \cdot u_0 ) - ( u_{n_a - 1} + a \cdot u_{k-1} ) \\
    & =   u_{n_a} - u_{n_a - 1} - a \cdot u_{k-1}
  \end{align*}
  whence 
  \begin{align*}
    \nth{\diff{\kappa}}{n} 
    &=
    \begin{cases}
      u_1 - u_0 & \text{if $n \equiv 1 \pmod{k}$} \\
      u_2 - u_1 & \text{if $n \equiv 2 \pmod{k}$} \\
      \vdots & \vdots \\
      u_{k-1} - u_{k-2} & \text{if $n \equiv k-1 \pmod{k}$} \\
       \nth{\diff{\kappa}}{n/k} - u_{k-1} & \text{if $n \equiv 0 \pmod{k}$} 
    \end{cases}
  \end{align*}
  From Lemma~\ref{lem:toeplitz:recurrence} we infer that 
  this is exactly the recurrence equation of the Toeplitz word 
  generated by the pattern $\Tdiff{u} = \diff{u} \gap^{-u_{k-1}}$.
\end{proof}

\begin{example}
  The first difference of the Morse sequence~$\morse$
  is the conjugate of the \pd\ sequence~$\PD$:
  \(
    \diff{\morse} 
    = \diff{\keane{01}} 
    = \toeplitz{\cat{\diff{01}}{\rotgap{1}}}
    = \toeplitz{\cat{1}{\rotgap{1}}} 
    = \PD+1
  \).
  
  The first difference of the Mephisto Waltz~$\mephisto$~\cite{kean:1968}:  
  \(
    \diff{\mephisto} 
    = \diff{\keane{001}} 
    = \toeplitz{\diff{001}{\rotgap{1}}} 
    = \toeplitz{\cat{01}{\rotgap{1}}} 
  \)
  turns out to be the sequence of turns (folds) of the alternate Terdragon curve~\cite{davi:knut:1970}\,!
  Moreover we also know $\diff{\mephisto} = \div{\sierpinski}{2}$, 
  where $\sierpinski = \toeplitz{00\rotgap{1} 11\rotgap{1}}$ is the 
  sequence of turns of the \Sierpinski{} curve~\cite{endr:hend:klop:2009,endr:hend:klop:2011}\,!
  Note that $\toeplitz{\cat{01}{\rotgap{1}}}$ 
  is the additive sequence $\pgs{\residues{2}{3}\cup\{3\}}$ 
  derived from $\dlog{3,2}$, see the construction in Section~\ref{sec:toeplitz:additive}.
\end{example}

We conclude with a complete characterization of the arithmetic self-similarity of the Thue--Morse sequence~$\morse$, 
and leave the characterization of the entire class of Keane words 
as future work. 
We know that at least $\pair{i}{\length{u}^n} \in \AS{\keane{u}}$ 
for all $n\ge 0$ and $0 \le i \lt \length{u}^n$ by Lemma~\ref{lem:keane:nth}.

\subsection*{The Arithmetic Self-Similarity of the \thuemorse{} Sequence} 

\begin{theorem}\label{thm:AS:morse}
  \( \AS{\morse} = \displset{\pair{a}{b}}{0 \leq a \lt b = 2^m \text{ for some $m \geq 0$}} \)\,.
\end{theorem}

We have the following recurrence equations for $\morse = \keane{01} = 0110 1001 \dots$:
\begin{align}
  \morse(2n) = \morse(n) && \morse(2n+1) = \inv{\morse(n)}
  \label{eq:morse:recurrence}
\end{align}
\emph{Carry-free addition} $\alpha \uplus \beta$ of binary numbers 
$\alpha = \alpha_p \ldots \alpha_0$ and $\beta = \beta_q \ldots \beta_0$ 
is defined by:
\[
  \alpha \uplus \beta 
  =
  \begin{cases}
  \alpha + \beta & \text{if there is no $i \leq \min(p,q)$ with $\alpha_i = \beta_i = 1$} \\
  \text{\textit{undefined}} & \text{otherwise}
  \end{cases}
\]
Carry-free addition preserves the total number of ones: 
$|\alpha\uplus\beta|_1 = |\alpha|_1 + |\beta|_1$.
Hence we have that $\morse(x+y) = \morse(x) + \morse(y)$ 
whenever $(x)_2 \uplus (y)_2$ is defined.
\begin{lemma}\label{lem:joerg:20100310}
  Let $x > 0$ with $x \neq 2^k$ for any $k \geq 0$. 
  Then $\morse(xy)=1$ for some $y > 0$ with $\morse(y) = 0$.
\end{lemma}
\proof
  Without loss of generality, we let $x$ be an odd number.
  Otherwise, if $x=2x'$, then $\morse(2x'y)=\morse(x'y)$.
  Put differently, trailing $0$'s of $(x)_2$ can be removed, 
  as they do not change the parity of the number of $1$'s in $(xy)_2$.
  
  Let $(x)_2 = x_n \ldots x_1 x_0$ be the binary representation of $x$ (with $x_n = x_0 = 1$).
  Note that $n > 0$ for otherwise $x = 2^0$.
  We distinguish two cases:
  \begin{enumerate} 
    \item 
    If $x_1 = 0$, then we take $(y)_2 = y_n \ldots y_0$ with $y_n = y_0 = 1$ and $y_i = 0$ for $0 < i < n$.
    So $y = 2^n + 2^0$ and $xy = x2^n + x$\,; 
    in binary representation:
    \[
    \begin{array}{cccccccc}
          &        &     & x_n & x_{n-1} & \ldots & x_1 & x_0 \\
          &        &     & 1   & 0       & \ldots & 0   & 1   \\
          \hline 
          &        &     & x_n & x_{n-1} & \ldots & x_1 & x_0 \\
      x_n & \ldots & x_1 & x_0 & 0       & \ldots & 0   & 0   \\
      \hline 
      x_n & \ldots & 1 & 0   & x_{n-1} & \ldots & x_1 & x_0
    \end{array}
    \]
    To see that $\morse(xy) = 1$, i.e., that $(xy)_2$ has an odd number of $1$'s,
    observe that $(xy)_2$ (the bottom line) contains the subword $x_{n-1} \ldots x_2$ twice. 
    These cancel each other out, 
    and, as $x_1 = 0$, the only remaining $1$'s are $x_n$ and $x_0$ on the outside 
    and the created $1$ at position $n+1$. That makes three.    
    
    \item
    In case $x_1 = 1$, 
    we take $y = 3 \prim{y}$ with $\prim{y}$ the result of case (i) for $3x$.
    First of all note that $(3x)_2$ ends in $01$ and so case (i) applies. 
    Secondly, we get that $\morse(3x\prim{y}) = 1$, and hence $\morse(xy) = 1$.
    Finally, to see that $\morse(y) = 0$, note that $\prim{y} = 2^m + 1$ for some $m > 0$,
    and so $\morse(y) = \morse(3\prim{y}) = \morse(2^{m+1} + 2^m + 2 + 1) = 1 - \morse(2^m + 2^{m-1} + 1)$
    by the recurrence equations~\eqref{eq:morse:recurrence} for $\morse$.
    To see that $\morse(2^m + 2^{m-1} + 1) = 1$, one distinguishes cases $m=1$ (then $\morse(4)=1$) and $m > 1$ 
    (then we again have three $1$'s in the binary expansion).
    \qed
  \end{enumerate} 

\proof[Proof of Theorem~\ref{thm:AS:morse}]
  The direction~`$\supseteq$' follows from Lemma~\ref{lem:keane:nth}.
  
  For the direction~`$\subseteq$', assume $\pair{a}{b}\in\AS{\morse}$.
  We distinguish two cases: (i)~$b$ is a power of $2$ and $a \geq b$,
  and (ii)~$b$ is not a power of $2$, and show that they lead to a contradiction. 
  \begin{enumerate}
  \item
  Let $m \geq 0$ and $a \geq b = 2^m$.
  Consider $k\geq 1$ such that $bk \leq a \lt b(k+1)$
  and let $\prim{a} = a - bk$. 
  Then
  $\asub{\morse}{a}{b}(n) = \morse(a+bn) = \morse(a'+b(k+n)) = \asub{\morse}{\prim{a}}{b}(k+n)$,
  in short: $\asub{\morse}{a}{b} = \asub{(\asub{\morse}{\prim{a}}{b})}{k}{1}$~($\ast$).
  As we have that $\prim{a} \lt b$,  
  it follows from 
  ($\supseteq$) above that $\pair{\prim{a}}{b}\in\AS{\morse}$, and so  
  $\asub{\morse}{\prim{a}}{b} \similarx \morse$~($\ast\ast$).
  Combining ($\ast$) and ($\ast\ast$) with the assumption $\asub{\morse}{a}{b}\similarx\morse$,
  we obtain $\asub{\morse}{k}{1} \similarx \morse$.
  
  Suppose $\asub{\morse}{s}{1} \eq \morse$ for some $s\geq 1$ 
  (the proof of $\asub{\morse}{s}{1} \neq \inv{\morse}$ is analogous).%
  \footnote{%
    We may also conclude this case by using the fact that $\morse$ is non-periodic:
    No positive iteration $k \ge 1$ of the shift operator of $\morse$ equals $\morse$ or its conjugate,
    because a sequence $\astr\in\str{A}$ such that $\asub{\astr}{k}{1} = \astr$ with $k \geq 1$ is $k$-periodic, 
    see Lemma~\ref{lem:suffix:in:AS:periodic}.%
  }
  This means $\morse(s + n) = \morse(n)$ for all $n$.
  So $\morse(s) = \morse(0) = 0$.
  Now let $p$ be maximal such that $2^p \leq s$
  and $(s)_2 = s_p s_{p-1} \ldots s_0$ (so $s_p = 1$).
  Then $(s + 2^p)_2 = 1 0 s_{p-1} \ldots s_0$,
  and hence $\morse(s+2^p) = \morse(s) = 0$.
  But the assumption $\asub{\morse}{s}{1} \eq \morse$ 
  tells us $\morse(s+2^p) = \morse(2^p) = 1$.

  \item
  Assume $b > 0$ is not a power of $2$, and let $a \geq 0$ be arbitrary.
  By Lemma~\ref{lem:joerg:20100310} there exists $y \gt 0$ 
  such that $\morse(y) = 0$ and $\morse(by) = 1$.
  Let $m$ be minimal such that $a < 2^m$. 
  Then also $\morse(b y 2^m) = 1$, and, since $(a + by2^m)_2 = (a)_2 \uplus (by2^m)_2$,
  we get $\morse(a + by2^m) = \morse(a) + 1 \pmod 2$.
  
  From the assumption $\pair{a}{b}\in\AS{\morse}$
  it follows that either $\myall{n}{\morse(a+bn)=\morse(n)}$,
  or $\myall{n}{\morse(a+bn)=\inv{\morse}(n)}$.
  In the first case we have $\morse(a) = \morse(0) = 0$
  and $\morse(a + by2^m) = \morse(y2^m) = 0$.
  In the second case, we have $\morse(a) = \inv{\morse}(0) = 1$
  and $\morse(a + by2^m) = \inv{\morse}(y2^m) = 1$.
  Both cases contradict $\morse(a + by2^m) = \morse(a) + 1 \pmod 2$.
  \qed
  \end{enumerate}

\bibliography{main}

\begin{thebibliography}{10}

\bibitem{allo:bach:1992}
J.-P. Allouche and R.~Bacher.
\newblock {Toeplitz Sequences, Paperfolding, Towers of Hanoi, and
  Progression-Free Sequences of Integers}.
\newblock {\em L'Enseignement Math\'ematique}, 38:315--327, 1992.

\bibitem{allo:dres:1990}
J.-P. Allouche and F.~Dress.
\newblock {Tours de Hano\"{\i} et Automates}.
\newblock {\em Informatique Th\'{e}orique et Applications}, 24:1--16, 1990.

\bibitem{allo:shal:1999}
J.-P. Allouche and J.~Shallit.
\newblock {The Ubiquitous Prouhet--Thue--Morse Sequence}.
\newblock In {\em Proceedings of International Conference on Sequences and
  Their Applications (SETA~1998)}, pages 1--16. Springer, 1999.

\bibitem{allo:shal:2003}
J.-P. Allouche and J.~Shallit.
\newblock {\em Automatic Sequences: Theory, Applications, Generalizations}.
\newblock Cambridge University Press, New York, 2003.

\bibitem{cass:fere:maud:riva:sark:2000}
J.~Cassaigne, S.~Ferenczi, C.~Mauduit, J.~Rivat, and A.~S\'{a}rk\"{o}zy.
\newblock {On Finite Pseudorandom Binary Sequences IV: The Liouville Function,
  II}.
\newblock {\em Acta Arithmetica}, 95:349--359, 2000.

\bibitem{cass:karh:1997}
J.~Cassaigne and J.~Karhum{\"a}ki.
\newblock {Toeplitz Words, Generalized Periodicity and Periodically Iterated
  Morphisms}.
\newblock {\em European Journal of Combinatorics}, 18(5):497--510, 1997.

\bibitem{conw:1976}
J.H. Conway.
\newblock {\em {On Numbers and Games}}.
\newblock London Mathematical Society Monographs (No.~6). Academic Press, 6
  edition, 1976.

\bibitem{dann:2011}
F.G.W. Dannenberg.
\newblock {Symbolic Dynamics and Automata Theory}.
\newblock Master's thesis, Delft University of Technology, 2011.

\bibitem{davi:knut:1970}
C.~Davis and D.E. Knuth.
\newblock {Number Representations and Dragon Curves}.
\newblock {\em Journal of Recreational Mathematics}, 3:66--81, 133--149, 1970.

\bibitem{diri:1837}
P.G.L. Dirichlet.
\newblock {Beweis des Satzes, da\ss\ jede unbegrenzte arithmetische
  Progression, deren erstes Glied und Differenz ganze Zahlen ohne
  gemeinschaftlichen Factor sind, unendlich viele Primzahlen enth\"alt}.
\newblock {\em Abhandlungen Akademische Wissenschaft Berlin}, 48:45--81, 1837.

\bibitem{endr:hend:klop:2009}
J.~Endrullis, D.~Hendriks, and J.W. Klop.
\newblock {Let's Make a Difference!}

\bibitem{endr:hend:klop:2011}
J.~Endrullis, D.~Hendriks, and J.W. Klop.
\newblock {Degrees of Streams}.
\newblock {\em Integers, Electronic Journal of Combinatorial Number Theory},
  2011.
\newblock Accepted for publication.

\bibitem{hare:coon:2011}
K.~Hare and M.~Coons.
\newblock \textit{Personal communication}, 2011.

\bibitem{jaco:1992}
K.~Jabobs.
\newblock {\em Invitation to Mathematics}.
\newblock Princeton University Press, 1992.

\bibitem{jaco:kean:1969}
K.~Jacobs and M.~Keane.
\newblock {$0$-$1$-Sequences of Toeplitz Type}.
\newblock {\em Zeitschrift f\"{u}r Wahrscheinlichkeitstheorie und Verwandte
  Gebiete}, 13(2):123--131, 1969.

\bibitem{kean:1968}
M.~Keane.
\newblock {Generalized Morse Sequences}.
\newblock {\em Zeitschrift f\"{u}r Wahrscheinlichkeitstheorie und Verwandte
  Gebiete}, 10(4):335--353, 1968.

\bibitem{sloa:2010}
N.J.A. Sloane.
\newblock {Online Encyclopedia of Integer Sequences}.
\newblock {\small{\url{http://oeis.org/}}}.

\end{thebibliography}

\end{document}